\documentclass[11pt,a4paper,reqno]{amsart}
\usepackage[T1]{fontenc}
\usepackage[utf8]{inputenc}
\usepackage{lmodern}
\usepackage{amsmath,amssymb,mathtools}
\usepackage{graphicx}
\usepackage[expansion=false]{microtype}
\usepackage[a4paper,left=23mm,right=23mm,top=23mm,bottom=25mm,headheight=13pt,headsep=8mm]{geometry}
\usepackage{enumitem}
\usepackage[hidelinks,bookmarksnumbered]{hyperref}

\makeatletter
\renewcommand{\subsection}{\@startsection{subsection}{2}%
  \z@{.5\linespacing\@plus.7\linespacing}{.25\linespacing}%
  {\normalfont\bfseries}}
\makeatother

\numberwithin{equation}{section}

\newtheorem{theorem}{Theorem}[section]
\newtheorem{proposition}[theorem]{Proposition}
\newtheorem{lemma}[theorem]{Lemma}
\newtheorem{corollary}[theorem]{Corollary}
\newtheorem*{theoremA}{Theorem A}
\newtheorem*{theoremB}{Theorem B}
\newtheorem*{theoremC}{Theorem C}
\newtheorem*{theoremD}{Theorem D}
\theoremstyle{definition}
\newtheorem{definition}[theorem]{Definition}
\theoremstyle{remark}
\newtheorem{remark}[theorem]{Remark}

\newcommand{\C}{\mathbb C}
\newcommand{\Z}{\mathbb Z}
\newcommand{\Q}{\mathbb Q}
\newcommand{\Hh}{\mathbb H}
\newcommand{\Gammaone}{\mathrm{SL}_2(\mathbb Z)}
\newcommand{\M}{M_*}
\newcommand{\wt}{\operatorname{wt}}
\newcommand{\tr}{\operatorname{tr}}
\newcommand{\Span}{\operatorname{span}}
\newcommand{\OO}{\mathrm O}
\newcommand{\DD}{\mathcal D}
\newcommand{\e}{\mathrm e}
\newcommand{\cE}{\mathcal E}
\newcommand{\cO}{\mathcal O}
\newcommand{\cV}{\mathcal V}
\newcommand{\ve}{\mathbf e}
\newcommand{\dual}{\vee}
\newcommand{\Sym}{\operatorname{Sym}}
\newcommand{\scriptsum}{\mathop{\vcenter{\hbox{\scalebox{0.55}{$\textstyle\sum$}}}}\nolimits}

\allowdisplaybreaks[2]
\begin{document}

\title[Free modules of vector-valued modular forms]{Free modules of vector-valued modular forms and applications to Jacobi forms}

\author{Jialin Li}

\address{School of Mathematics and Statistics, Wuhan University, Wuhan 430072, Hubei, China}

\email{jlli.math@whu.edu.cn}

\author{Haowu Wang}

\address{School of Mathematics and Statistics, Wuhan University, Wuhan 430072, Hubei, China}

\email{haowu.wangmath@whu.edu.cn}

\author{Bohan Zhang}

\address{School of Mathematics and Statistics, Wuhan University, Wuhan 430072, Hubei, China}

\email{bhzhang.math@whu.edu.cn}


\date{\today}

\begin{abstract}
The space $M(\rho;\mathcal{E})$ of weakly holomorphic modular forms for a dimension-$r$ complex representation $\rho$ of $\mathrm{SL}_2(\mathbb{Z})$ with growth prescribed by $\mathcal{E}$ at the cusp is a free module over $\C[E_4,E_6]$. We first prove that a $\C[E_4,E_6]$-linearly independent homogeneous family $F_1,\ldots,F_r$ in this module gives a basis if and only if $\sum_{i=1}^r\wt(F_i)=12\tr\cE$. We then extend this criterion to Jacobi forms. As applications, we determine the generator weights of free modules of $\mathrm{O}(D_{8d}^+)$-invariant Jacobi forms of index $D_{8d}^{+}(p)$, where $D_{8d}^{+}$ is an even unimodular lattice extending $D_{8d}$, $d\ge2$, and $p=2,3,5$. 
\end{abstract}
\maketitle

\section{Introduction}

Let $\rho:\,\Gammaone\to\operatorname{GL}(V)$ be a complex representation
with $\dim V=r$.
Assume that $\rho(T)$ has finite order, with
$T=\begin{psmallmatrix}1&1\\0&1\end{psmallmatrix}$.
A vector-valued weakly holomorphic modular form is a holomorphic function $F:\,\Hh\to V$ of integral weight $k$ that satisfies
\[
F(\gamma\tau)=(c\tau+d)^k\rho(\gamma)F(\tau),
\qquad \gamma=\begin{pmatrix}a&b\\c&d\end{pmatrix}\in\Gammaone,
\]
and has a Fourier expansion with exponents bounded from below. 
Write $\e(x)=e^{2\pi i x}$ and $q=\e(\tau)$. 

A complete system of exponent conditions
$\cE=\{(\ell_i,c_i):1\le i\le r\}$ consists of a basis
$\ell_1,\ldots,\ell_r$ of the dual space $V^\dual$ and $c_i\in\Q$ such that $\ell_i\circ\rho(T)=\e(c_i)\ell_i$. 
Denote by $M_k(\rho;\cE)$ the space of modular forms of weight $k$ for $\rho$ satisfying $\ell_i(F)=O(q^{c_i})$ for every $i$, and put
$\tr\cE=\sum_i c_i$.

Vector-valued modular forms play an essential role in the theory of Borcherds products and vertex operator algebras.
The inputs of Borcherds products are weakly holomorphic vector-valued
modular forms for Weil representations; see \cite{Bo98, Br02}.
Another source is Zhu's modular invariance theorem for characters of suitable vertex operator algebras \cite{Zhu96,DLM00}. 
More precisely, the characters are transformed with zero weight and have
expansions of the form $q^{h_i-c/24}\sum_{j\ge0}a_{i,j}q^j$,
where $c$ is the central charge and $h_i$ is the conformal weight
of the corresponding irreducible module.
The leading exponents $h_i-c/24$ can be negative.
The spaces $M_k(\rho;\cE)$ permit such componentwise lower bounds on Fourier exponents to be prescribed. 

When $\rho=\mathbf1$ and $\cE=0$, the spaces $M_k(\rho;\cE)$ become 
classical spaces of scalar modular forms.
Their graded ring is $\M:=\bigoplus_k M_k=\C[E_4,E_6]$.
Multiplication by scalar modular forms makes
$M(\rho;\cE):=\bigoplus_k M_k(\rho;\cE)$ a graded $\M$-module.

For the standard holomorphic cusp condition, the freeness over $\M$
was proved by Marks and Mason \cite[Theorem~1]{MM10}. 
Candelori and Franc \cite[Theorem~5.5]{CF17} gave a geometric proof of
freeness and showed that the sum of the generator weights is equal to twelve times the trace of the exponent matrix.
Their argument also gives freeness for nonstandard exponent choices
\cite[Remark~5.6]{CF17}.
Franc and Mason \cite{FM18} established some bounds on generator weights and
their multiplicities.
For standard holomorphic exponents, the generator weights for irreducible
representations of dimensions two and three were determined by Marks
\cite[Theorems~3.1 and~4.1]{Marks11}.
Freeness, the generator-weight sum, and the determinant formula
for general exponent choices are stated in
\cite[Theorem~3.4]{Gan14}.
We give direct analytic proofs adapted to the prescribed cusp
conditions considered here, complementing the geometric approach
of Candelori--Franc.
In particular, we establish the following weight-sum criterion
for a homogeneous free basis. 

\begin{theoremA}
Let $\rho:\, \Gammaone\to\operatorname{GL}(V)$ be a complex representation of dimension $r$.
Assume that $\rho(T)$ has finite order and
$\rho(-I)=(-1)^\kappa I$, with $\kappa\in\{0,1\}$.
For every complete system of exponent conditions $\cE$, the graded module
$M(\rho;\cE)=\bigoplus_{k\equiv\kappa\, (2)} M_k(\rho;\cE)$
is free of rank $r$ over $\M$.
For a homogeneous form $F$, write $\wt(F)$ for its weight.
If $F_1,\ldots,F_r\in M(\rho;\cE)$ are $\M$-linearly independent and
homogeneous, then
\[
\sum_{j=1}^r\wt(F_j)\ge12\tr\cE.
\]
The equality holds if and only if $F_1,\ldots,F_r$ form an $\M$-basis.
\end{theoremA}

Note that for homogeneous forms $F_1,\ldots,F_r\in M(\rho;\cE)$,
$\M$-linear independence is equivalent to
$\det(F_1,\ldots,F_r)\not\equiv0$;
see Remark~\ref{rem:determinant-independence}.

We now extend this criterion to Jacobi forms and present some precise applications. 
Let $L$ be a positive definite even lattice of even rank $n$, and
let $G$ be a subgroup of $\OO(L)$.
An index-$L$ Jacobi form $\varphi:\,\Hh\times(L\otimes\C)\to \mathbb{C}$ 
is called $G$-invariant if 
$\varphi(\tau,gz)=\varphi(\tau,z)$ for every $g\in G$.
Write $J_{k,L}^{G}$ and $J_{k,L}^{\mathrm{w},G}$ for the spaces of
$G$-invariant holomorphic and weak Jacobi forms of weight $k$, respectively, and put $J_{*,L}^{G}=\bigoplus_kJ_{k,L}^{G}$ and
$J_{*,L}^{\mathrm{w},G}=\bigoplus_kJ_{k,L}^{\mathrm{w},G}$. The theory of Jacobi forms was first introduced by Eichler--Zagier
\cite{EZ85}, and later generalized to lattice-indices by Gritsenko
\cite{Gri88}. 

Let $L^\dual$ be the dual lattice and put $Q_L(x)=(x,x)_L/2$.
The theta coefficient vector of a $G$-invariant Jacobi form takes
values in $\C[L^\dual/L]^G$.
It transforms under the restriction of the contragredient Weil
representation, denoted by $\rho_L^G$.
This space has a basis consisting of the sums
$\sum_{\gamma\in\cO}\ve_\gamma$ over the $G$-orbits
$\cO\subset L^\dual/L$, where $\ve_\gamma$ are the standard basis vectors.
For $\gamma\in\cO$, the weak condition imposes the lower bound
$-\min_{x\in\gamma+L}Q_L(x)$ on Fourier exponents, while the holomorphic
condition imposes the least nonnegative rational number congruent to
$-Q_L(\gamma)$ modulo $\Z$.
Together with the dual of this basis, these bounds give complete
systems $\cE_{L,G}^{\mathrm{w}}$ and $\cE_{L,G}^{\mathrm{hol}}$, respectively. Consequently, theta decomposition gives 
\[
\begin{aligned}
J_{k,L}^{\mathrm{w},G}&\cong M_{k-n/2}(\rho_L^G;\cE_{L,G}^{\mathrm{w}}),\\
J_{k,L}^{G}&\cong M_{k-n/2}(\rho_L^G;\cE_{L,G}^{\mathrm{hol}}).
\end{aligned}
\]
Thus, the generator weights of $G$-invariant Jacobi form modules are those of the corresponding vector-valued modular form modules increased by $n/2$. 

Explicit module structures of Jacobi forms have been studied for several classes of lattices. For root lattices, generators of Weyl invariant Jacobi form modules have been studied in \cite{EZ85, Wir92, AG20, Wang21a, Wang21, SW23, Sakai26, WW20}. 
For the Leech lattice $\Lambda$, generators of Conway invariant Jacobi form modules are known for lattice indices $\Lambda(t)$ with
$1\le t\le3$; see \cite{SW22}. For rank-two lattices, explicit generators and dimension formulas for weak Jacobi forms are given in \cite{WW23}. For these results, the generator weights are determined by determining the ring structure of Jacobi forms or by estimating the dimensions of the spaces of low-weight Jacobi forms. These methods are difficult to apply to general lattices. In this paper, we develop a systematic and efficient approach to this problem. 

Applying Theorem~A through the theta decomposition yields the following result. Its proof is given in Theorem~\ref{thm:invariant-theta} and
Corollary~\ref{cor:invariant-basis}. 

\begin{theoremB}
Put $r=|G\backslash(L^\dual/L)|$. Both modules $J_{*,L}^{\mathrm{w},G}$ and
$J_{*,L}^{G}$ are free of rank $r$ over $\M$.
Then an $\M$-linearly independent homogeneous family
$\varphi_1,\ldots,\varphi_r\in J_{*,L}^{\mathrm{w},G}$ is a basis if and only if
\[
\sum_{i=1}^r\wt(\varphi_i)=\frac{rn}{2}+12\tr\cE_{L,G}^{\mathrm{w}};
\]
an $\M$-linearly independent homogeneous family
$\psi_1,\ldots,\psi_r\in J_{*,L}^{G}$ is a basis if and only if
\[
\sum_{i=1}^r\wt(\psi_i)=\frac{rn}{2}+12\tr\cE_{L,G}^{\mathrm{hol}}.
\]
\end{theoremB}

The theorem extends straightforwardly to odd-rank lattices when $\mathrm{SL_2(\mathbb{Z})}$ is replaced by $\mathrm{Mp_2(\mathbb{Z})}$. 

As applications, we apply the theorem to lattices $D_{8d}^{+}(p)$,
where $d\ge2$ and $p=2,3,5$. Define
\[
D_{8d}=\{x\in\Z^{8d}:\textstyle\sum_i x_i\equiv0\pmod2\},
\quad h=\tfrac12(1,\ldots,1),
\quad D_{8d}^{+}=D_{8d}\cup(h+D_{8d}).
\]
With the standard Euclidean bilinear form, $D_{8d}^{+}$ is even unimodular, and $\OO(D_{8d}^{+})$ coincides with the Weyl group $W(D_{8d})$ of $D_{8d}$. Put $G_d=W(D_{8d})$.
For a positive integer $m$, $D_{8d}^+(m)$ denotes the same underlying
group with its bilinear form multiplied by $m$.
Write $J_{*,D_{8d}^+(m)}^{G_d}$ and $J_{*,D_{8d}^+(m)}^{\mathrm{w},G_d}$ for
the modules of $G_d$-invariant holomorphic and weak Jacobi forms, respectively. 

For lattice index $D_{8d}^{+}(2)$, let $\rho_d$ be the coefficient
representation on $\cV_d$. We prove that
\[
\cV_d=V_d^{\mathrm{cyc}}\oplus U_d,\qquad
\dim V_d^{\mathrm{cyc}}=2d+1,\qquad \dim U_d=3.
\]
The weak Jacobi form of weight $-4$ on $D_8$ in
\cite[Section~3.1.2]{AG20} gives a theta coefficient vector
$\Phi_d$ of weight $-8d$ whose values span $V_d^{\mathrm{cyc}}$.
Its successive Serre derivatives give the basis in
Corollary~\ref{cor:cyclic-basis}. We prove that projection to $U_d$
is surjective with the induced exponent conditions. Hence
\[
M(\rho_d;\cE_d^{\mathrm{w}})
\cong
\left(\bigoplus_{j=0}^{2d}\M\,\mathcal D^j\Phi_d\right)
\oplus M(\rho_d^U;\cE_d^U),
\]
where $\rho_d^U=\rho_d|_{U_d}$.
Theorem~\ref{thm:rank3-classification} determines the generator weights
for the direct sum of a one-dimensional representation and an irreducible
two-dimensional representation, with exponent conditions that need not
respect this decomposition. Applying it to $M(\rho_d^U;\cE_d^U)$ and
adding the theta weight $4d$ gives the weak case of the following result.

For the holomorphic case, we use the principal-part obstruction criterion
of \cite[Theorem~3.1]{Bo99} and \cite[Theorem~1.17]{Br02}, in the form of Proposition~\ref{prop:principal-parts}, to determine the dimensions above the singular weight $4d$ from the weak spaces.
At singular weight, matching upper and lower bounds gives 
\[
\dim M_0(\rho_d)=\left\lfloor\frac d3\right\rfloor+2.
\]
The freeness then determines the weights of the holomorphic generators and their multiplicities.

\begin{theoremC}
For $d\ge2$, the $\M$-module $J_{*,D_{8d}^+(2)}^{\mathrm{w},G_d}$ is free
with generator-weight multiset
\[
\{-4d+2j:0\le j\le2d\}\sqcup\{-8,-6,-4\};
\]
and $J_{*,D_{8d}^+(2)}^{G_d}$ is free, with generator
weights and multiplicities given in Theorem~\ref{thm:holomorphic-weights}.
\end{theoremC}

For lattice indices $D_{8d}^{+}(3)$ and $D_{8d}^{+}(5)$,
the invariant coefficient Weil representations factor through $A_4$
and $A_5$, respectively.
We compute their irreducible multiplicities from the character tables.
The canonical holomorphic exponent conditions respect these decompositions.
Applying the results of Section~\ref{sec:rankthree} to the irreducible
summands and increasing their generator weights by $4d$ gives the following result.

\begin{theoremD}
For $d\ge2$, both $\M$-modules $J_{*,D_{8d}^+(3)}^{G_d}$ and
$J_{*,D_{8d}^+(5)}^{G_d}$ are free with generator weights in
$\{4d,\, 4d+2,\, 4d+4,\, 4d+6,\, 4d+8\}$.
Their multiplicities are given in
Theorems~\ref{thm:holomorphic-index-three} and
\ref{thm:holomorphic-index-five}, respectively. 
\end{theoremD}

The remainder of this paper is organized as follows.
Section~\ref{sec:prescribed} proves Theorem~A. 
Section~\ref{sec:rankthree} determines the generator weights for $\rho$ of low dimensions.
Section~\ref{sec:jacobi} relates $G$-invariant Jacobi forms to vector-valued
modular forms through theta decomposition and proves Theorem~B.
Section~\ref{sec:Dfamily} treats the lattice index $D_{8d}^{+}(2)$ and proves Theorem~C.
Section~\ref{sec:odd-indices} treats lattice indices $D_{8d}^{+}(3)$ and
$D_{8d}^{+}(5)$ and proves Theorem~D.

\bigskip
\noindent
\textbf{Acknowledgements}. This work was supported by the National Key R$\&$D Program of China (Grant No. 2024YFA1014500). H. Wang thanks Kaiwen Sun for helpful discussions. 

\bigskip
\noindent
\textbf{Declaration of AI Use.} 
The authors used AI tools to assist with language polishing and proofreading, and to perform some computer computations in the final section.

\section{Vector-valued modular forms with prescribed growth}\label{sec:prescribed}

We prove freeness and the weight-sum basis criterion for vector-valued
modular forms with prescribed growth at the cusp. We also deduce
a tensor-product formula.

\subsection{Vector-valued modular forms}

We define the prescribed cusp conditions and establish Wronskian and
determinant bounds. Fix the standard generators of $\mathrm{SL}_2(\mathbb{Z})$: 
\[
S=\begin{pmatrix}0&-1\\1&0\end{pmatrix},
\qquad
T=\begin{pmatrix}1&1\\0&1\end{pmatrix}.
\]
For an integer $k$ and a holomorphic function $F:\Hh\to V$, write
\[
(F|_k\gamma)(\tau)=(c\tau+d)^{-k}F(\gamma\tau),
\qquad
\gamma=\begin{pmatrix}a&b\\c&d\end{pmatrix}\in\Gammaone.
\]
Put $q=\e(\tau)=e^{2\pi i\tau}$.  Throughout Sections~\ref{sec:prescribed} and \ref{sec:rankthree}, $(\rho,V)$ denotes a finite-dimensional complex representation of $\Gammaone$ such that $\rho(T)$ has finite order.  Assume
\[
\rho(-I)=(-1)^\kappa I,
\qquad\kappa\in\{0,1\},
\]
and use only weights $k\equiv\kappa\pmod2$.  A general representation is the direct sum of its two $(-I)$-eigenspaces, so all statements apply componentwise.

We consider holomorphic functions $F:\Hh\to V$ satisfying
\[
F|_k\gamma=\rho(\gamma)F
\qquad \text{for any $\gamma\in\Gammaone$},
\]
and having Fourier expansions with exponents bounded below in $\frac1N\Z$ for some positive integer $N$.  We write $M_k(\rho)$ for the space of such functions with no negative Fourier exponents. A form in $M_k(\rho)$ is a \emph{cusp form} if all its nonzero Fourier terms have strictly positive exponents. We put $M(\rho)=\bigoplus_kM_k(\rho)$. For Fourier expansions and
vector-valued Poincar\'e series, see \cite{KM03,KM04}.  For a homogeneous form $F$ of weight $k$, we write $\wt(F)=k$.  In the scalar case,
\[
M_k=M_k(\mathbf1),\qquad \M=\bigoplus_k M_k=\C[E_4,E_6].
\]
Here $E_j$ denotes the normalized Eisenstein series of even weight $j\ge4$, and $E_2$ is the normalized quasimodular Eisenstein series of weight two.  We use
\[
\eta(\tau)=q^{1/24}\prod_{n\ge1}(1-q^n),\qquad
\Delta=\eta^{24}=\frac{E_4^3-E_6^2}{1728}.
\]
Let $\nu_\eta^2$ denote the character defined by
$\eta^2|_1\gamma=\nu_\eta^2(\gamma)\eta^2$.
For an integer $m$, the notation $\nu_\eta^{2m}$ means $(\nu_\eta^2)^m$.
Every one-dimensional character of $\Gammaone$ is a power of $\nu_\eta^2$ and is determined by its value on $T$.

\begin{definition}\label{def:exponent-conditions}
A system of \emph{exponent conditions} for $\rho$ is a finite collection
\[
\cE=\{(\ell_i,c_i)\}_{i=1}^m
\]
with $\ell_i\in V^\dual\setminus\{0\}$ and $c_i\in\Q$ such that
\[
\ell_i\circ\rho(T)=\e(c_i)\ell_i.
\]
It is \emph{complete} if $m=\dim V$ and $\ell_1,\ldots,\ell_m$ is a basis of $V^\dual$.  We define
\[
M_k(\rho;\cE)
=
\{F:F|_k\gamma=\rho(\gamma)F\text{ for all }\gamma\in\Gammaone,
\ \ell_i(F)=O(q^{c_i})\text{ for every }i\},
\]
where $F$ is holomorphic on $\Hh$ and has Fourier exponents bounded below.
For a complete system, put
\[
\tr\cE=\sum_i c_i.
\]
\end{definition}

Redundant conditions can be removed as follows.  Order the pairs so that
$c_1\ge\cdots\ge c_m$.  Retain $(\ell_i,c_i)$ if $\ell_i$ does not lie in the span of the previously retained covectors.  The retained pairs form a system $\cE'$ satisfying
\begin{equation}\label{eq:reduced-exponent-conditions}
|\cE'|=m':=\dim\Span_{\C}\{\ell_1,\ldots,\ell_m\},
\qquad M_k(\rho;\cE')=M_k(\rho;\cE).
\end{equation}
Indeed, each discarded $\ell_i$ is a linear combination of retained covectors whose bounds are at least $c_i$, so their conditions imply $\ell_i(F)=O(q^{c_i})$.  The reverse inclusion follows since every retained condition belongs to $\cE$.  In particular, if the original covectors span $V^\dual$, the retained system is complete. Except when removing redundant conditions, all systems used below are complete.

For a nonzero homogeneous form $F$, set
\[
V(F)=\Span_{\C}\{F(\tau):\tau\in\Hh\}.
\]

\begin{lemma}\label{lem:value-span}
The space $V(F)$ is $\rho$-invariant.  If
\[
F=\sum_{j=1}^m f_jv_j
\]
with $v_1,\ldots,v_m$ linearly independent and $f_1,\ldots,f_m$ linearly independent over $\C$, then $\dim V(F)=m$.
\end{lemma}

\begin{proof}
For $\gamma\in\Gammaone$ and $\tau\in\Hh$,
\[
\rho(\gamma)F(\tau)=(c\tau+d)^{-k}F(\gamma\tau),
\]
so $V(F)$ is invariant.  Since $V(F)\subset\Span_{\C}\{v_1,\ldots,v_m\}$, we have $\dim V(F)\le m$.  If $\dim V(F)<m$, there is a nonzero functional $\ell$ on this latter span that vanishes on $V(F)$. Hence
\[
0=\ell(F(\tau))=\sum_{j=1}^m\ell(v_j)f_j(\tau)
\qquad(\tau\in\Hh).
\]
At least one $\ell(v_j)$ is nonzero, since the $v_j$ form a basis of the span.  This contradicts the linear independence of the $f_j$.  Thus $\dim V(F)=m$.
\end{proof}

Let $F\ne0$ and put $m=\dim V(F)$.  Restrict the covectors in a complete system
\[
\cE=\{(\ell_i,c_i)\}_{i=1}^r
\]
to $V(F)$ and discard zero restrictions.  The remaining covectors span $V(F)^\dual$.  Apply \eqref{eq:reduced-exponent-conditions} and write the retained system as
\[
\cE_F=\{(\ell_{i_j}|_{V(F)},c_{i_j}):1\le j\le m\}.
\]
Then $\cE_F$ is complete and describes exactly the conditions imposed by $\cE$ on forms taking values in $V(F)$.  We compute $W(F)$ in the basis dual to its retained covectors.

For a homogeneous form of weight $k$, the Serre derivative is
\[
\DD_kF=q\frac{dF}{dq}-\frac{k}{12}E_2F.
\]
For $j\ge1$ set
\[
\DD^jF=\DD_{k+2j-2}\circ\cdots\circ\DD_kF,
\qquad \DD^0F=F.
\]
We usually write simply $\DD F$.  It has weight $k+2$, obeys the Leibniz rule, and preserves every prescribed lower bound on Fourier exponents.  Thus
\[
\DD:M_k(\rho;\cE)\longrightarrow M_{k+2}(\rho;\cE).
\]
Since $q\,d\log\eta/dq=E_2/24$, one also has
\[
\DD_kF=0
\quad\Longleftrightarrow\quad
\frac{d}{d\tau}(\eta^{-2k}F)=0
\quad\Longleftrightarrow\quad
F=\eta^{2k}v
\]
for a constant vector $v\in V$.  If $F\ne0$, its transformation law then gives
$\rho(\gamma)v=\nu_\eta^{2k}(\gamma)v$.  In particular, the value span of a nonzero form killed by $\DD$ is a one-dimensional subrepresentation.

\begin{proposition}\label{prop:derivative-independent}
Let $F\ne0$ be homogeneous and put $m=\dim V(F)$.  Then
\[
F,\DD F,\ldots,\DD^{m-1}F
\]
are linearly independent over the field of meromorphic functions on $\Hh$, and hence over $\M$.
\end{proposition}

\begin{proof}
We use the leading-coefficient calculation for the modular Wronskian; compare \cite[Section~3]{Mason07}.  Choose coordinates on $V(F)$, and let $f_1,\ldots,f_m$ be the component functions.  They are linearly independent over $\C$, since a nontrivial constant relation would place $V(F)$ in a proper coordinate hyperplane.  Eliminating their lowest Fourier coefficients successively gives a basis $g_1,\ldots,g_m$ of their span with
\[
g_i=a_iq^{\lambda_i}+O(q^{\lambda_i+1/N}),
\qquad a_i\ne0,\qquad \lambda_1<\cdots<\lambda_m,
\]
where $N$ is a common denominator of the Fourier exponents.  At each step the kernel of the lowest-coefficient functional has codimension one and has strictly larger lowest Fourier exponent.  Thus this elimination is possible.

Put $k=\wt(F)$ and
\[
P_0(X)=1,\qquad
P_j(X)=\prod_{a=0}^{j-1}\left(X-\frac{k+2a}{12}\right).
\]
Since $E_2=1+O(q)$, the coefficient of $q^{\lambda_i}$ in $\DD^jg_i$ is $a_iP_j(\lambda_i)$.  Each $P_j$ is monic of degree $j$, so
\[
\det(\DD^{j-1}g_i)_{1\le i,j\le m}
=q^{\scriptsum_i\lambda_i}\left(
\prod_{i=1}^m a_i
\prod_{1\le i<j\le m}(\lambda_j-\lambda_i)
+O(q^{1/N})\right).
\]
The displayed coefficient is nonzero.  The change from $f_i$ to $g_i$ is an invertible constant change of coordinates, so $\det(F,\DD F,\ldots,\DD^{m-1}F)\not\equiv0$.  Hence the columns are linearly independent over the field of meromorphic functions on $\Hh$.
\end{proof}

The following determinant calculation will be used below.

\begin{lemma}\label{lem:determinant}
Let $\cE=\{(\ell_i,c_i)\}_{i=1}^r$ be complete and let $F_1,\ldots,F_r$ be homogeneous forms of weights $k_1,\ldots,k_r$ in $M(\rho;\cE)$.  Then $12\tr\cE\in\Z$ and
\[
\frac{\det(F_1,\ldots,F_r)}{\eta^{24\tr\cE}}
\in M_{\scriptsum_i k_i-12\tr\cE}.
\]
Here the determinant is computed with respect to the basis of $V$ dual to $\ell_1,\ldots,\ell_r$.
\end{lemma}

\begin{proof}
The determinant transforms with the one-dimensional character $\det\rho$.  Since
\[
\det\rho(T)=\e(\tr\cE),
\]
and every character of $\Gammaone$ takes a twelfth root of unity at $T$, we have $12\tr\cE\in\Z$.  The characters $\det\rho$ and $\nu_\eta^{24\tr\cE}$ agree at $T$, and hence
\[
\det\rho=\nu_\eta^{24\tr\cE}.
\]  The $i$th row of the determinant is $O(q^{c_i})$, so
\[
\det(F_1,\ldots,F_r)=O(q^{\tr\cE}).
\]
Since $\eta$ has no zeros on $\Hh$ and $\eta^{24\tr\cE}=q^{\tr\cE}(1+O(q))$, the quotient is holomorphic on $\Hh$ and at the cusp.  Its character is trivial and its weight is the asserted one.
\end{proof}

Write $F=\sum_{i=1}^m f_i v_i$ in the basis of $V(F)$ dual to the
covectors in $\cE_F$, and put
\[
d=q\frac{d}{dq}=\frac{1}{2\pi i}\frac{d}{d\tau}.
\]
The Wronskian of the component functions is
$\det(d^{j-1}f_i)_{1\le i,j\le m}$.
Since $\DD^j$ differs from $d^j$ by a linear combination of
lower-order derivatives, column operations give
\[
\det(d^{j-1}f_i)_{1\le i,j\le m}
=\det(F,\DD F,\ldots,\DD^{m-1}F).
\]
We use the normalized Wronskian
\[
W(F):=
\frac{\det(d^{j-1}f_i)_{1\le i,j\le m}}
{\eta^{24\tr\cE_F}}.
\]
By Proposition~\ref{prop:derivative-independent} and
Lemma~\ref{lem:determinant}, it is a nonzero scalar modular form of weight
\begin{equation}\label{eq:wronskian-weight}
\wt W(F)=m\,\wt(F)+m(m-1)-12\tr\cE_F.
\end{equation}

\subsection{The free-module theorem}

In this subsection, we prove freeness for prescribed growth at the cusp.
The standard holomorphic case is proved in \cite[Theorem~1]{MM10}.
A geometric proof, also valid for nonstandard exponent choices, is given in
\cite[Theorem~5.5 and Remark~5.6]{CF17}.
We adapt the divisibility argument of \cite[Section~3]{MM10} to the
prescribed cusp conditions. For convenience, we present the full proof below.

\begin{theorem}\label{thm:freeness}
Let $(\rho,V)$ and $\cE$ be as above, with $\cE$ complete and $r=\dim V$.  Then
\[
M(\rho;\cE)=\bigoplus_{k\equiv\kappa\, (2)}M_k(\rho;\cE)
\]
is a free $\M$-module of rank $r$.
\end{theorem}

\begin{proof}
It is enough to treat $\rho(-I)=I$.  Indeed, if $\rho(-I)=-I$, division by $\eta^2$ changes odd weights to even weights and replaces $\rho$ by $\rho\otimes\nu_\eta^{-2}$ while shifting every prescribed lower bound on Fourier exponents by $-1/12$.  Multiplication by $\eta^2$ recovers the original module.

\medskip\noindent\textup{(1)} We prove that $M_k(\rho;\cE)=0$ for all sufficiently small $k$ and that each graded piece is finite-dimensional.  Let
\[
c_{\min}=\min_i c_i.
\]
For $0\ne F\in M_k(\rho;\cE)$, every lower bound in $\cE_F$ is at least $c_{\min}$.  The nonzero modular Wronskian and \eqref{eq:wronskian-weight} therefore give
\begin{equation}\label{eq:lower-bound-freeness}
k\ge12c_{\min}-\dim V(F)+1\ge12c_{\min}-(r-1).
\end{equation}
Thus $M(\rho;\cE)$ is bounded below.  Each graded piece is finite-dimensional: otherwise, for fixed $k$ one could find, for every $N$, a nonzero $F\in M_k(\rho;\cE)$ vanishing to order at least $N$ at $i$.  Since $E_6$ has a simple zero at the $\Gammaone$-orbit of $i$ and $E_6=1+O(q)$ at the cusp, the quotient $F/E_6^N$ would lie in $M_{k-6N}(\rho;\cE)$, contradicting \eqref{eq:lower-bound-freeness} for large $N$.

\medskip\noindent\textup{(2)} We prove freeness.  In each admissible weight $k$, choose homogeneous lifts $\{F_i\}$ of a $\C$-basis of
\begin{equation}\label{eq:indecomposable-quotient}
M_k(\rho;\cE)\big/
\bigl(E_4M_{k-4}(\rho;\cE)+E_6M_{k-6}(\rho;\cE)\bigr).
\end{equation}
Since the module is bounded below, these lifts generate $M(\rho;\cE)$ by induction on the weight.

They are also $\M$-linearly independent.  Suppose that there is a nonzero homogeneous relation of least total weight $K$,
\[
\sum_j g_jF_j=0.
\]
By the choice of the lifts in \eqref{eq:indecomposable-quotient}, no coefficient of weight $0$ can occur in such a minimal relation.  Hence every nonzero $g_j$ lies in $(E_4,E_6)$; write
\[
g_j=E_4a_j+E_6b_j.
\]
With
\[
A=\sum_j a_jF_j,\qquad B=\sum_j b_jF_j,
\]
we have $E_4A=-E_6B$.  Since $E_4$ and $E_6$ have no common zero on $\Hh$, the common quotient
\[
C:=\frac{A}{E_6}=-\frac{B}{E_4}
\]
is holomorphic on $\Hh$; moreover $E_4,E_6=1+O(q)$, so $C$ satisfies the same cusp conditions.  Its weight is $K-10<K$, and therefore the already constructed generators of smaller weight give
\[
C=\sum_j c_jF_j,
\qquad c_j\in\M.
\]
It follows that
\[
\sum_j(a_j-E_6c_j)F_j=0,
\qquad
\sum_j(b_j+E_4c_j)F_j=0.
\]
Both relations have smaller total weight, so minimality gives
\[
a_j=E_6c_j,\qquad b_j=-E_4c_j,
\]
and hence $g_j=0$ for every $j$, a contradiction.  Thus the chosen generators form a free homogeneous generating set.

\medskip\noindent\textup{(3)} We prove that the rank is $r$.  Let $k_0$ be a lower bound for the nonzero graded pieces.  Given large $k$, choose $N\ge0$ such that
\[
k_0-12\le k-12N<k_0.
\]
Multiplication by $\Delta^N$ gives
\[
M_k(\rho;\cE)=\Delta^N M_{k-12N}(\rho;\cE-N),
\qquad
\cE-N=\{(\ell_i,c_i-N)\}.
\]
Since $M_{k-12N}(\rho;\cE)=0$, the $Nr$ Fourier coefficients between the bounds $c_i-N$ and $c_i$ define an injection
\[
M_{k-12N}(\rho;\cE-N)\hookrightarrow\C^{Nr}.
\]
Consequently
\[
\dim M_k(\rho;\cE)\le Nr=\frac r{12}k+O(1).
\]
For any $s$ distinct members $F_1,\ldots,F_s$ of this free generating set, their scalar multiples give
\[
\dim M_k(\rho;\cE)\ge
\sum_{j=1}^s\dim M_{k-\wt(F_j)}
=\frac s{12}k+O(1)
\]
along the admissible parity.  Hence $s\le r$.  In particular, the free generating set is finite and has at most $r$ elements.

There are at least $r$ independent forms.  Choose a $T$-eigenbasis $v_1,\ldots,v_r$ and representatives $\alpha_i\in(0,1]$ with
\[
\rho(T)v_i=\e(\alpha_i)v_i.
\]
Put $\Gamma_\infty=\langle T,-I\rangle$.  For sufficiently large even $k$, consider the Poincar\'e series
\[
P_i(\tau)=\sum_{\gamma\in\Gamma_\infty\backslash\Gammaone}
 (c_\gamma\tau+d_\gamma)^{-k}\rho(\gamma)^{-1}
 \e\bigl((\alpha_i-1)\gamma\tau\bigr)v_i,
\qquad
\gamma=\begin{pmatrix}a_\gamma&b_\gamma\\c_\gamma&d_\gamma\end{pmatrix}.
\]
The convergence theorem and the Fourier expansion in
\cite[Theorems~3.3 and~3.10, equation~(41)]{KMlog} apply: here $\rho(T)$ is diagonal, so all Jordan blocks have size one and the logarithmic correction matrix in that source is the identity.  Choose its integer parameters to obtain exponents $\alpha_i-1$ in the diagonal exponential factors.  The series converge absolutely and locally uniformly on $\Hh$, have weight $k$ for $\rho$, and satisfy
\[
(P_i)_j=\delta_{ij}q^{\alpha_i-1}+q^{\alpha_j}O(1).
\]
Factoring $q^{\alpha_j-1}$ from row $j$ gives
\[
\det(P_1,\ldots,P_r)
=q^{\scriptsum_{j=1}^r\alpha_j-r}\det(I_r+O(q))
=q^{\scriptsum_{j=1}^r\alpha_j-r}(1+O(q)).
\]
Thus $P_1,\ldots,P_r$ are $\M$-linearly independent.  After multiplication by one sufficiently large common power of $\Delta$, all of them satisfy the prescribed growth conditions $\cE$, and their independence is unchanged.  Therefore the rank is at least $r$.

The free generating set consequently has exactly $r$ elements.
\end{proof}

\begin{remark}\label{rem:determinant-independence}
The proof also shows that every homogeneous basis has a nonzero determinant.
Indeed, express the $r$ forms constructed from the Poincar\'e series in such
a basis. Their determinant is the product of the basis determinant and a
scalar coefficient determinant, and is nonzero.
Consequently, $r$ homogeneous forms in $M(\rho;\cE)$ are $\M$-linearly
independent if and only if their determinant is not identically zero.
To see this, write the forms in a homogeneous basis, work over the fraction
field of $\M$, and clear denominators.
\end{remark}

For general restrictions on generator weights and their multiplicities,
see \cite{FM18}.

\subsection{A basis criterion and tensor products}

We give a criterion for a homogeneous free basis in terms of the sum
of its weights and apply it to tensor products.
The corresponding generator-weight sum and basis-determinant formula
are stated by Gannon
\cite[Theorem~3.4(b)]{Gan14}.
Here we give a different proof, based on determinant identities and
the Serre derivative, adapted to the prescribed cusp conditions
considered in this paper.
These identities yield the weight-sum criterion for a homogeneous
free basis.
The determinant argument uses the following fact about the Serre
derivative of scalar modular forms.

%We give a criterion for a homogeneous free basis in terms of the sum
%of its weights and apply it to tensor products. The determinant argument
%uses the following fact about the Serre derivative of scalar modular forms.

\begin{lemma}\label{lem:kernel-serre}
Let $0\ne f\in M_w$ be a scalar modular form.  Then $\mathcal D_wf=0$ if and only if
\[
f=c\Delta^m
\]
for some $c\in\C^\times$ and $m\ge0$.
\end{lemma}

\begin{proof}
For every integer $w$,
\[
\DD_wf=0\quad\Longleftrightarrow\quad f=c\eta^{2w}
\]
by the derivative identity above.  If $c\ne0$ and $f$ has trivial character, then $\nu_\eta^{2w}(T)=\e(w/12)=1$, so $w=12m$ for an integer $m$.  Holomorphy at the cusp gives $m\ge0$.  Conversely, $\DD_{12}\Delta=0$ and the Leibniz rule show that every $c\Delta^m$ lies in the kernel.
\end{proof}

\begin{theorem}\label{thm:basis-criterion}
Let $F_1,\ldots,F_r\in M(\rho;\cE)$ be $\M$-linearly independent and homogeneous.  Then
\[
\{F_1,\ldots,F_r\}\text{ is an }\M\text{-basis}
\quad\Longleftrightarrow\quad
\sum_{i=1}^r\wt(F_i)=12\tr\cE.
\]
Equivalently,
\[
\frac{\det(F_1,\ldots,F_r)}{\eta^{24\tr\cE}}\in\C^\times.
\]
\end{theorem}

\begin{proof}
Set
\[
h=\frac{\det(F_1,\ldots,F_r)}{\eta^{24\tr\cE}}.
\]
By Lemma~\ref{lem:determinant}, $h$ is a nonzero scalar modular form of weight
\[
\sum_i\wt(F_i)-12\tr\cE.
\]
In particular, $\sum_i\wt(F_i)\ge12\tr\cE$.
Assume that $F_1,\ldots,F_r$ form a basis.  We show that $h$ is constant.

Assume first that $\mathcal Dh\ne0$.  Since $\mathcal D\eta^{24\tr\cE}=0$, the Leibniz rule for determinants gives
\[
\mathcal D\det(F_1,\ldots,F_r)
=
\sum_{j=1}^r
\det(F_1,\ldots,\mathcal DF_j,\ldots,F_r)
=
\eta^{24\tr\cE}\mathcal Dh.
\]
Hence at least one determinant on the right is nonzero.  For this index $j$, the family obtained by replacing $F_j$ with $\mathcal DF_j$ is $\M$-linearly independent.  Expanding $\mathcal DF_j$ in the given basis, we obtain
\[
\mathcal DF_j=\sum_i a_iF_i,
\qquad
 a_i\in M_{\wt(F_j)+2-\wt(F_i)}.
\]
The coefficient of $F_j$ has weight $2$ and is therefore zero.  Thus $\mathcal DF_j$ belongs to the span of the other $F_i$, contradicting the nonzero determinant.

Thus $\mathcal Dh=0$.  By Lemma~\ref{lem:kernel-serre},
\[
h=c\Delta^m.
\]
Suppose $m>0$.  Choose an integer $K$ of the correct parity, sufficiently large that for every $j$ there is a scalar modular form $A_j\in M_{K-\wt(F_j)}$ with constant term $1$.  We may take $A_j=1$ when the weight is $0$ and a normalized Eisenstein series otherwise; choose $K$ so that no difference is $2$.  The forms
\[
G_j=A_jF_j
\]
have the same weight $K$.  Write $F_j=\sum_i f_{ij}v_i$ in the basis dual to the exponent covectors.  Since
\[
\frac{\det(G_1,\ldots,G_r)}{\eta^{24\tr\cE}}
=c\Delta^m\prod_jA_j
\]
vanishes at $q=0$, so does $q^{-\tr\cE}\det(G_1,\ldots,G_r)$.  Hence the matrix of constant terms of the normalized Fourier series
\[
\bigl(q^{-c_i}\ell_i(G_j)\bigr)_{i,j}
\]
is singular.  Hence there are constants $\lambda_j$, not all zero, such that
\[
G:=\sum_j\lambda_jG_j
\]
satisfies every cusp condition one order more strongly:
\[
\ell_i(G)=O(q^{c_i+1})
\qquad(1\le i\le r).
\]
Therefore $G/\Delta\in M_{K-12}(\rho;\cE)$.  Since $F_1,\ldots,F_r$ form a basis, write
\[
\frac{G}{\Delta}=\sum_j b_jF_j,
\qquad b_j\in\M.
\]
Multiplication by $\Delta$ and uniqueness of the expansion in the basis $F_j$ give
\[
\lambda_jA_j=\Delta b_j
\]
for every $j$.  For an index with $\lambda_j\ne0$, this is impossible since $A_j$ has constant term $1$.  Thus $m=0$, and $h$ is constant.

Conversely, suppose $h=c\in\C^\times$.  Let $F\in M_k(\rho;\cE)$.  Cramer's rule gives
\[
F=\sum_i a_iF_i,
\qquad
 a_i=
\frac{\det(F_1,\ldots,F_{i-1},F,F_{i+1},\ldots,F_r)}
{\det(F_1,\ldots,F_r)}.
\]
By Lemma~\ref{lem:determinant}, the numerator divided by $\eta^{24\tr\cE}$ is a scalar modular form of weight $k-\wt(F_i)$, while the denominator divided by the same eta-power is the nonzero constant $c$.  Hence $a_i\in M_{k-\wt(F_i)}\subset\M$.  Therefore $F_1,\ldots,F_r$ is a basis.

A nonzero holomorphic scalar modular form has weight zero exactly when it is constant.  Hence the determinant condition is equivalent to the asserted equality of weight sums.
\end{proof}

The basis criterion also gives the following tensor-product statement.  Let $\cE_a=\{(\ell_i^{(a)},c_i^{(a)}):1\le i\le r_a\}$ be complete systems for $(\rho_a,V_a)$, where $r_a=\dim V_a$ and $a=1,2$.  Define
\[
\cE_1\boxtimes\cE_2
=
\{(\ell_i^{(1)}\otimes\ell_j^{(2)},c_i^{(1)}+c_j^{(2)}):1\le i\le r_1,\ 1\le j\le r_2\}.
\]

\begin{corollary}\label{cor:tensor}
The map induced by the pointwise tensor product $(F\otimes G)(\tau)=F(\tau)\otimes G(\tau)$,
\[
M(\rho_1;\cE_1)\otimes_{\M}M(\rho_2;\cE_2)
\longrightarrow
M(\rho_1\otimes\rho_2;\cE_1\boxtimes\cE_2)
\]
is an isomorphism.
\end{corollary}

\begin{proof}
Let $F_1,\ldots,F_{r_1}$ and $G_1,\ldots,G_{r_2}$ be bases.  Write $A_F(\tau)$ and $A_G(\tau)$ for their matrices of component functions.  The component matrix of the products $F_i\otimes G_j$ is $A_F\otimes A_G$, and
\[
\det(A_F\otimes A_G)=(\det A_F)^{r_2}(\det A_G)^{r_1}\not\equiv0.
\]
Thus these $r_1r_2$ forms are independent.  Moreover,
\begin{align*}
\sum_{i,j}\wt(F_i\otimes G_j)
&=r_2\sum_i\wt(F_i)+r_1\sum_j\wt(G_j)\\
&=12r_2\tr\cE_1+12r_1\tr\cE_2\\
&=12\tr(\cE_1\boxtimes\cE_2).
\end{align*}
Theorem~\ref{thm:basis-criterion} finishes the proof.
\end{proof}

\section{Generator weights in low dimensions}\label{sec:rankthree}

We determine the generator weights for representations of small dimension.
These results will be used in Sections~\ref{sec:Dfamily} and~\ref{sec:odd-indices}.

\subsection{Irreducible representations}

We first determine bases and generator weights for irreducible
representations of dimensions at most three.
For standard holomorphic exponents, cyclicity in dimensions two and three
is proved in \cite[Theorems~3.1 and~4.1]{Marks11}. Generator-weight results for irreducible representations of
dimension less than six are given in
\cite[Sections~4.2--4.3]{Gan14}.
We include direct proofs of the cases needed below.

\begin{lemma}\label{lem:rank-one}
Let $\chi$ be one-dimensional and let $\cE=\{(\ell,x)\}$ be complete.  Then
\[
M(\chi;\cE)=\M\,\eta^{24x}v,
\]
where $\ell(v)=1$.  In particular, the unique generator weight is $12x$ and
\[
\mathcal D(\eta^{24x}v)=0.
\]
\end{lemma}

\begin{proof}
The character of $\eta^{24x}$ has the same value on $T$ as $\chi$, hence is $\chi$.  The displayed form belongs to the prescribed module and has weight $12x$.  Theorem~\ref{thm:basis-criterion} shows that it is a basis.  The last assertion follows from $q\,d\log\eta/dq=E_2/24$.
\end{proof}

\begin{lemma}\label{lem:rank-two}
Let $\rho$ be irreducible of dimension two and let $\cE$ be complete, with
\[
s=\tr\cE.
\]
Then there is a form $F$ of weight $6s-1$ such that
\[
F,\mathcal DF
\]
is a basis of $M(\rho;\cE)$.  Thus the generator weights are
\[
6s-1,\qquad6s+1.
\]
\end{lemma}

\begin{proof}
Let $k_1\le k_2$ be the weights of a homogeneous basis and choose a nonzero form $F$ of least weight $k_1$.  Irreducibility implies $V(F)=V$, so $F$ and $\mathcal DF$ are independent by Proposition~\ref{prop:derivative-independent}.  Therefore a generator of weight at most $k_1+2$ is needed.  Since the weights have fixed parity, either $k_2=k_1$ or $k_2=k_1+2$.

If $k_2=k_1$, choose independent $F,G$ of weight $k_1$.  Their derivatives have weight $k_1+2$, but $M_2=0$, so a basis of weights $k_1,k_1$ would force $\mathcal DF=\mathcal DG=0$.  This would give one-dimensional invariant value spans, contrary to irreducibility.  Hence $k_2=k_1+2$, and we may take the second generator to be $\mathcal DF$.  Theorem~\ref{thm:basis-criterion} gives
\[
2k_1+2=12s,
\]
so $k_1=6s-1$.
\end{proof}

The same argument in rank three gives the irreducible cyclic case.

\begin{lemma}\label{lem:rank-three-irreducible}
If $\rho$ is irreducible of dimension three, then
\[
M(\rho;\cE)=\M F\oplus\M\mathcal DF\oplus\M\mathcal D^2F
\]
for a form $F$ of weight $4\tr\cE-2$.
\end{lemma}

\begin{proof}
Let $k_0$ be the least nonzero weight and let $m=\dim M_{k_0}(\rho;\cE)$.  A basis of this space extends to a homogeneous free basis, since no nonzero form of weight $k_0$ lies in $(E_4,E_6)M(\rho;\cE)$.  Hence $1\le m\le3$.

If $m=3$, all three free generators have weight $k_0$.  Their derivatives must vanish, since $M_2=0$, contradicting irreducibility.  If $m=2$, choose independent $F,G$ of weight $k_0$ and let $H$ be the third generator.  The derivatives $\DD F,\DD G$ are nonzero.  Thus $\wt(H)=k_0+2$, and
\[
\DD F=aH,\qquad\DD G=bH
\]
for constants $a,b$, not both zero.  The nonzero form $bF-aG$ is killed by $\DD$ and has a one-dimensional invariant value span, again a contradiction.  Therefore $m=1$.

Choose $F$ spanning the least-weight space.  Irreducibility gives $V(F)=V$, so Proposition~\ref{prop:derivative-independent} shows that $F,\DD F,\DD^2F$ are independent.  If $k_1\le k_2$ are the remaining generator weights, then $\DD F\ne0$ and $M_2=0$ give $k_1=k_0+2$; choose the corresponding generator to be $\DD F$.  If $k_2=k_0+2$ or $k_2>k_0+4$, every form of weight $k_0+4$ is a multiple of $E_4F$, since $M_2=0$.  This is impossible for $\DD^2F$.  Hence $k_2=k_0+4$, and its coefficient in $\DD^2F$ is a nonzero constant.  Replacing the third generator by $\DD^2F$ gives a basis.  Theorem~\ref{thm:basis-criterion} now gives $3k_0+6=12\tr\cE$.
\end{proof}

The following two lists of generator weights will be used for Jacobi forms of lattice
index $D_{8d}^{+}(5)$.

\begin{proposition}\label{prop:rank-four-five}
Let $\rho$ be irreducible, with $\rho(-I)=I$, and let $\cE$ be a
complete system of exponent conditions with $\tr\cE=2$.
If $\dim\rho=4$, the generator weights of $M(\rho;\cE)$ are
\[
4,\quad6,\quad6,\quad8.
\]
If $\dim\rho=5$, they are
\[
2,\quad4,\quad4,\quad6,\quad8.
\]
\end{proposition}

\begin{proof}
Write $r=\dim\rho$.  For every nonzero homogeneous form $F$,
irreducibility gives $V(F)=V$.  Proposition~\ref{prop:derivative-independent}
therefore gives a nonzero Wronskian, and
\eqref{eq:wronskian-weight} implies
\[
r\bigl(\wt(F)+r-1\bigr)\ge24.
\]
The Serre derivative is injective, since a nonzero form killed by it
has a one-dimensional invariant value span.  Theorem~\ref{thm:basis-criterion}
gives a total generator weight of $24$, and all weights are even.

Suppose first that $r=4$.  The Wronskian bound gives $\wt(F)\ge4$.
The possible nondecreasing generator-weight lists are
\[
(4,4,4,12),\quad(4,4,6,10),\quad(4,4,8,8),\quad
(4,6,6,8),\quad(6,6,6,6).
\]
For the first three lists, $\dim M_4(\rho;\cE)>\dim M_6(\rho;\cE)$,
contrary to injectivity of $\DD_4$.  For the last list,
$M_6(\rho;\cE)$ has dimension four and $M_8(\rho;\cE)=0$.
Thus only $(4,6,6,8)$ remains.

Now suppose that $r=5$.  The Wronskian bound gives $\wt(F)\ge2$.
Let $a$ and $b$ be the numbers of generators of weights two and four.
Since the scalar space $M_2$ is zero, injectivity of $\DD_2$ gives
$b\ge a$.
If $a=0$, the only possible lists are
$(4,4,4,4,8)$ and $(4,4,4,6,6)$.
Both contradict injectivity of $\DD_4$.
If $a\ge3$, the inequality $b\ge a$ contradicts the rank.
If $a=2$, the weight sum forces $(2,2,4,4,12)$.
Choose a basis $F_1,F_2$ of $M_2(\rho;\cE)$.
Their derivatives form a basis of $M_4(\rho;\cE)$, whereas
\[
M_6(\rho;\cE)=E_4\Span_{\C}\{F_1,F_2\}.
\]
Consequently, $\DD^2F_i$ belongs to the meromorphic span of
$F_1,F_2$.  The meromorphic span of
$F_1,F_2,\DD F_1,\DD F_2$ is therefore preserved by $q\,d/dq$,
and hence contains every successive Serre derivative of $F_1$.
Its dimension is at most four, contradicting the independence of
$F_1,\DD F_1,\ldots,\DD^4F_1$.

It remains that $a=1$.  The possible lists are
\[
(2,4,4,4,10),\qquad(2,4,4,6,8),\qquad(2,4,6,6,6).
\]
The first contradicts injectivity of $\DD_4$.
For the last, let $F$ span the weight-two space.  Then $\DD F$
spans the weight-four space, and
\[
M_8(\rho;\cE)=\Span_{\C}\{E_6F,E_4\DD F\}.
\]
Thus $\DD^3F$ lies in the meromorphic span of $F,\DD F$,
again contradicting Proposition~\ref{prop:derivative-independent}.
The only remaining list is $(2,4,4,6,8)$.
\end{proof}

\subsection{Reducible representations of dimension three}\label{subsec:rank3classification}

We determine the possible generator weights for completely reducible
representations of dimension three. The main case is a direct sum of a
one-dimensional representation and an irreducible two-dimensional
representation, with cusp conditions that need not respect this decomposition.

Suppose first that $\rho$ is a direct sum of three one-dimensional characters.  Characters of $\Gammaone$ are determined by their value on $T$, so each $T$-eigenspace is a subspace on which $\rho$ acts by one character.  If $v_1,v_2,v_3$ is the basis dual to the covectors in $\cE=\{(\ell_i,c_i)\}_{i=1}^3$, each line $\C v_i$ is therefore invariant.  The forms
\[
\eta^{24c_1}v_1,\qquad\eta^{24c_2}v_2,\qquad\eta^{24c_3}v_3
\]
form a basis by Lemma~\ref{lem:rank-one}, applied to the three coordinate conditions.

For the remaining case, let $V=V_1\oplus V_2$ be an invariant direct sum, with
\[
\rho_j=\rho|_{V_j},\qquad \dim V_1=1,\qquad \dim V_2=2,
\]
and assume that $\rho_2$ is irreducible.  Fix a complete system
$\cE=\{(\ell_i,c_i):1\le i\le3\}$.

The two eigenvalues of $\rho_2(T)$ are distinct.  Indeed, if $\rho_2(T)=tI$, the relation $S^2=(ST)^3$ would imply $\rho_2(S)=t^{-3}I$, contrary to irreducibility.  Choose a nonzero vector $v_1\in V_1$ and a $T$-eigenbasis $v_2,v_3$ of $V_2$.  Define the covectors
\[
\varepsilon_1\in V_1^\dual,\qquad
\varepsilon_2,\varepsilon_3\in V_2^\dual
\]
by
\[
\varepsilon_1(v_1)=1,\qquad
\varepsilon_j(v_l)=\delta_{jl}\quad(j,l\in\{2,3\}),
\]
and set
\[
x_j=\max\{c_i:1\le i\le3,\ \ell_i(v_j)\ne0\}
\quad(j=1,2,3).
\]
The sets in these maxima are nonempty since the $\ell_i$ form a basis of $V^\dual$.  Put
\[
\cE_1=\{(\varepsilon_1,x_1)\},\qquad
\cE_2=\{(\varepsilon_2,x_2),(\varepsilon_3,x_3)\}.
\]
These are complete systems of exponent conditions for $\rho_1$ and $\rho_2$.

We check exactly which forms they describe.  For a form $G=g v_1$ taking values in $V_1$, the original conditions are
\[
\ell_i(v_1)g=O(q^{c_i})\qquad(1\le i\le3).
\]
They are equivalent to $g=O(q^{x_1})$.  Since $v_2,v_3$ have distinct $T$-eigenvalues, every nonzero restriction $\ell_i|_{V_2}$ is a scalar multiple of exactly one of $\varepsilon_2,\varepsilon_3$.  Thus a form $H=h_2v_2+h_3v_3$ taking values in $V_2$ satisfies the original conditions if and only if
\[
h_2=O(q^{x_2}),\qquad h_3=O(q^{x_3}).
\]
Identifying each summand with its image in $V$, we have therefore proved
\[
M(\rho_j;\cE_j)
=\{F\in M(\rho;\cE):F(\tau)\in V_j\text{ for all }\tau\in\Hh\}
\quad(j=1,2).
\]
Changing the nonzero scales of the chosen eigenvectors does not change $x_j$; interchanging $v_2,v_3$ only interchanges $x_2,x_3$.

The following lemma describes the possible bases in this case.

\begin{lemma}\label{lem:rank3-structure}
Let
\[
\rho=\rho_1\oplus\rho_2,
\qquad \dim\rho_1=1,
\qquad \dim\rho_2=2,
\]
where $\rho_2$ is irreducible.  Let $\cE$ be complete.  Then a homogeneous basis of $M(\rho;\cE)$ can be chosen in one of the following forms:
\begin{enumerate}[label=\textup{(\roman*)}]
\item $G_1,H,\mathcal DH$, where $G_1$ takes values in the one-dimensional summand;
\item $F_2,\mathcal DF_2,H$, where $F_2$ takes values in the two-dimensional summand;
\item $H,\mathcal DH,\mathcal D^2H$, where $V(H)=V$.
\end{enumerate}
\end{lemma}

\begin{proof}
Let $k_0$ be the least nonzero weight of $M(\rho;\cE)$ and choose a homogeneous free basis, using Theorem~\ref{thm:freeness}.

The space of weight-$k$ forms killed by $\DD$ has dimension at most one: such a form equals $\eta^{2k}v$, and $v$ must belong to the unique one-dimensional subrepresentation $V_1$.  Thus $\dim M_{k_0}(\rho;\cE)\ne3$, since three generators of weight $k_0$ would all have zero derivative.

Suppose $\dim M_{k_0}(\rho;\cE)=2$.  The two least-weight free generators $F,G$ form a basis of this space.  If the third generator had weight greater than $k_0+2$, then $\DD F=\DD G=0$, contrary to the preceding bound.  It therefore has weight $k_0+2$; call it $H$.  We have $\DD F=aH$ and $\DD G=bH$, with $(a,b)\ne(0,0)$.  Replacing one least-weight generator by the nonzero form $bF-aG$, and the other by a form with nonzero derivative, gives a basis of type~\textup{(i)}.

Now suppose $\dim M_{k_0}(\rho;\cE)=1$.  Write the chosen basis as $F,H,K$, with $\wt(F)=k_0<\wt(H)\le\wt(K)$.  Complete reducibility and the nonisomorphic irreducible summands $V_1,V_2$ imply that $V(F)$ is $V_1$, $V_2$, or $V$.

If $V(F)=V_1$, the coordinate description preceding the lemma and Lemma~\ref{lem:rank-one} show that $F$ is a nonzero constant multiple of the generator of $M(\rho_1;\cE_1)$.  Indeed, a nonconstant scalar multiple would have larger weight than that generator.  Consequently $P_2H\ne0$, where $P_2:V\to V_2$ is the equivariant projection: otherwise $H$ would be an $\M$-multiple of $F$.  The projection commutes with $\DD$.  If $\DD H$ belonged to $\M F+\M H$, the coefficient of $H$ would lie in $M_2=0$, giving $\DD(P_2H)=0$ and contradicting irreducibility of $\rho_2$.  Hence $K$ is needed in weight at most $\wt(H)+2$.  Equality $\wt(K)=\wt(H)$ is also impossible, since then the coefficients of both $H$ and $K$ in $\DD H$ would have weight two.  Thus $\wt(K)=\wt(H)+2$, and its coefficient in $\DD H$ is a nonzero constant.  Replacing $K$ by $\DD H$ gives type~\textup{(i)}.

If $V(F)=V_2$ or $V$, then $F,\DD F$ are independent.  Thus $\wt(H)=k_0+2$ and we may take $H=\DD F$.  If $V(F)=V_2$, this already gives type~\textup{(ii)}.  If $V(F)=V$, the form $\DD^2F$ is independent of $F,\DD F$.  As in Lemma~\ref{lem:rank-three-irreducible}, the third generator cannot have weight $k_0+2$ or a weight greater than $k_0+4$.  Its weight is therefore $k_0+4$, and it may be replaced by $\DD^2F$.  This gives type~\textup{(iii)}.
\end{proof}

Put
\[
s=\tr\cE,\qquad A=12x_1,\qquad B=6(x_2+x_3)-1.
\]
By Lemmas~\ref{lem:rank-one} and \ref{lem:rank-two}, choose generators
\[
G_1=\eta^{24x_1}v_1\in M_A(\rho_1;\cE_1),\qquad
F_2\in M_B(\rho_2;\cE_2),
\]
so that $G_1$ is a basis of the first module and $F_2,\DD F_2$ is a basis of the second.

\begin{theorem}\label{thm:rank3-classification}
With the preceding notation, put
\[
\delta=x_1+x_2+x_3-s\in\Z_{\ge0}.
\]
Then
\[
\delta=0
\quad\Longleftrightarrow\quad
M(\rho;\cE)=M(\rho_1;\cE_1)\oplus M(\rho_2;\cE_2).
\]

If $\delta=0$, the union of the two summand bases is a basis of $M(\rho;\cE)$, with weights
\[
A,\quad B,\quad B+2.
\]

Assume $\delta>0$.  Exactly one of the following three cases occurs:
\begin{enumerate}[label=\textup{(\arabic*)}]
\item If $A\le4s+2$, there is a basis $G_1,H,\mathcal DH$ of weights
\[
A,\quad 6(s-x_1)-1,\quad 6(s-x_1)+1.
\]
\item If $A>4s+2$ and $B\le4s$, there is a basis $F_2,\mathcal DF_2,H$ of weights
\[
B,\quad B+2,\quad 12(s-x_2-x_3).
\]
\item If $A>4s+2$ and $B>4s$, there is a cyclic basis $H,\mathcal DH,\mathcal D^2H$ of weights
\[
4s-2,\quad4s,\quad4s+2.
\]
\end{enumerate}
\end{theorem}

\begin{proof}
The preceding coordinate calculation shows that $G_1,F_2,\DD F_2$ belong to $M(\rho;\cE)$.  They are independent: the first takes values in $V_1$, while the last two are independent and take values in $V_2$.  Their total weight is
\[
A+B+(B+2)=12(x_1+x_2+x_3).
\]
Lemma~\ref{lem:determinant} gives a nonzero scalar modular form of weight $12\delta$, so $\delta\ge0$.  The eigenvalues of $T$ also give
\[
\e(x_1+x_2+x_3)=\det\rho(T)=\e(s),
\]
and hence $\delta\in\Z$.  By Theorem~\ref{thm:basis-criterion}, this particular independent family is a basis if and only if $\delta=0$.  This proves the displayed equivalence and the asserted generator weights when $\delta=0$.

Assume $\delta>0$.  Lemma~\ref{lem:rank3-structure} leaves three possible basis types.

We first identify the basis elements taking values in $V_1$ or $V_2$.  In type~(i), write the basis vector taking values in $V_1$ as $G=pG_1$, where $G_1$ generates $M(\rho_1;\cE_1)$.  If the other two basis vectors are $H,\mathcal DH$, then their projections to $V_2$ are $H_2,\mathcal DH_2$ and are independent by irreducibility and Proposition~\ref{prop:derivative-independent}.  Expanding $G_1$ in the full basis and projecting to $V_2$ therefore shows that the coefficients of $H$ and $\mathcal DH$ vanish.  Thus $G_1$ is an $\M$-multiple of $G=pG_1$, which forces $p\in\C^\times$.  Hence this basis vector may be chosen to be $G_1$ itself and has weight $A$.

Similarly, in type~(ii), write the basis as $F,\mathcal DF,H$ with $F$ taking values in $V_2$, and let $P_1:V\to V_1$ be the equivariant projection. Since $F,\mathcal DF,H$ form a basis, $P_1H\ne0$; otherwise their determinant would vanish identically. Every $G\in M(\rho_2;\cE_2)$ has an expansion
\[
G=aF+b\mathcal DF+cH,\qquad a,b,c\in\M.
\]
The projections of $G,F,\mathcal DF$ to $V_1$ are zero, so $cP_1H=0$ and hence $c=0$. Thus
\[
M(\rho_2;\cE_2)=\M F\oplus\M\mathcal DF.
\]
Their weights are therefore $B,B+2$ by Lemma~\ref{lem:rank-two}. Since $F$ has least weight in $M(\rho_2;\cE_2)$, it is a nonzero constant multiple of $F_2$. Rescaling gives the basis $F_2,\mathcal DF_2,H$.

In type~(i), the basis element taking values in $V_1$ is therefore $G_1$, of weight $A$.  The other two weights are $C,C+2$, where the weight-sum criterion gives
\[
C=6(s-x_1)-1=B-6\delta.
\]
The form of weight $C$ cannot take values only in $V_2$, since $C<B$ and $B$ is the least weight in $M(\rho_2;\cE_2)$.  It cannot take values only in $V_1$ either, since it is independent of $G_1$.  Thus its value span is three-dimensional.  The Wronskian bound \eqref{eq:wronskian-weight} gives
\[
3C+6\ge12s,
\]
which is equivalent to $A\le4s+2$.

In type~(ii), the first two forms are the restricted basis $F_2,\mathcal DF_2$ of weights $B,B+2$.  The third weight is
\[
C=12(s-x_2-x_3)=A-12\delta.
\]
The third form cannot take values in $V_1$, since $C<A$, and it cannot take values in $V_2$, since it is independent of the basis $F_2,\mathcal DF_2$ of $M(\rho_2;\cE_2)$.  Hence its value span is three-dimensional.  Again the Wronskian bound gives
\[
3C+6\ge12s.
\]
Using $C=12s-2B-2$ yields $B\le4s$.  Since $\delta\ge1$ and $C\ge4s-2$, the identity $C=A-12\delta$ also gives $A>4s+2$.

We next determine when a basis of type~(iii) can be replaced by a basis of type~(i) or~(ii).
In type~(iii), write the basis as $H,\mathcal DH,\mathcal D^2H$, with
$V(H)=V$. Its weights are
\[
C_0,\quad C_0+2,\quad C_0+4,
\qquad C_0=4s-2.
\]
The one-dimensional generator $G_1$ of weight $A$ is present in the module.  If $A<C_0+4$, then, since all weights have the same parity, $A=C_0$ or $A=C_0+2$ whenever such a form can occur.  The first possibility would make $G_1$ proportional to $H$, whose value span is $V$; the second would make it proportional to $\mathcal DH$ and hence not killed by $\mathcal D$.  Both are impossible.  If $A=C_0+4$, then $G_1$ lies in the two-dimensional space spanned by $E_4H$ and $\mathcal D^2H$.  The coefficient of $\mathcal D^2H$ is nonzero: otherwise $G_1$ would be a nonzero multiple of $E_4H$, whose value span is $V$, not $V_1$.  Replacing $\mathcal D^2H$ by $G_1$ therefore gives a basis of type~(i).  Hence a basis of type~(iii) with $A\le4s+2$ can also be written in type~(i), and this occurs only at $A=4s+2$.

Now assume $A>4s+2$.  The form $F_2$, of weight $B$, also belongs to $M(\rho;\cE)$.  In type~(iii), $B<4s-2$ is impossible, while $B=4s-2$ would make it proportional to $H$, whose value span is $V$.  If $B=4s$, it is proportional to $\mathcal DH$, and then $F_2,\mathcal DF_2,H$ is a basis of type~(ii).  Thus type~(iii) with $B\le4s$ can also be written as type~(ii).

Consequently, when $A\le4s+2$ a basis can always be chosen in type~(i).  When $A>4s+2$ and $B\le4s$, it can always be chosen in type~(ii).  When $A>4s+2$ and $B>4s$, only type~(iii) is possible.  These are the three cases in the theorem.
\end{proof}

\begin{remark}\label{rem:comparison-rank3}
For standard holomorphic exponents, irreducible representations of dimension less than four are cyclic under the modular derivative; see \cite{Marks11}.  Three-dimensional imprimitive representations and reducible but indecomposable representations are studied in \cite{FM16,CHMY}.  Theorem~\ref{thm:rank3-classification} addresses a different phenomenon: the representation itself is completely reducible, but the prescribed cusp growth need not respect the direct-sum decomposition.
\end{remark}

\section{Jacobi forms and Weil representations}\label{sec:jacobi}

Theta decomposition expresses a Jacobi form as a vector of modular
forms for the contragredient Weil representation. We describe explicitly
the exponent conditions corresponding to weak and holomorphic Jacobi
forms, including the shift in weight by half the lattice rank.
We then restrict to forms invariant under lattice automorphisms and
apply Section~\ref{sec:prescribed} to obtain freeness and weight-sum
basis criteria.

\subsection{Jacobi forms}

All lattices in this section are positive definite, even, and of even rank.  We use the lattice-index convention of \cite{Gri88}; the classical case is treated in \cite[Section~1]{EZ85}.

Let $L$ have even rank $n$.  Write
\[
Q_L(x)=\frac{(x,x)_L}{2},\qquad
L^\dual=\{x\in L\otimes\Q:(x,L)_L\subset\Z\},\qquad
A_L=L^\dual/L.
\]
Within this section we write $Q$ and $(\cdot,\cdot)$ for $Q_L$ and $(\cdot,\cdot)_L$.  On $A_L$, the same symbols denote their values in $\Q/\Z$.  For $x\in\Q$, write $\{x\}=x-\lfloor x\rfloor$.  Thus $\{-Q(\gamma)\}$ is independent of the lift chosen for $\gamma\in A_L$.

Let $\ve_\gamma$ be the standard basis of $\C[A_L]$, and let $\ve_\gamma^\dual$ be the dual covectors, with $\ve_\gamma^\dual(\ve_\delta)=\delta_{\gamma,\delta}$.  Give $\C[A_L]$ the Hermitian inner product for which this basis is orthonormal.  Define
\[
\theta_{L,\gamma}(\tau,z)
=\sum_{\ell\in L+\gamma}q^{Q(\ell)}\e((\ell,z)),
\qquad
\Theta_L(\tau,z)=\sum_{\gamma\in A_L}\theta_{L,\gamma}(\tau,z)\ve_\gamma.
\]
Positive definiteness gives normal convergence on compact subsets of $\Hh\times(L\otimes\C)$.  The theta vector transforms with the unitary Weil representation $\omega_L$:
\begin{align}
\omega_L(T)\ve_\gamma&=\e(Q(\gamma))\ve_\gamma,\label{eq:theta-T}\\
\omega_L(S)\ve_\gamma&=
\frac{\e(-n/8)}{\sqrt{|A_L|}}
\sum_{\delta\in A_L}\e(-(\gamma,\delta))\ve_\delta.
\label{eq:theta-S}
\end{align}
Here $n=\operatorname{rank}L$, as above.  Since $n$ is even, this representation is defined on $\Gammaone$.

A weak Jacobi form of weight $k$ and lattice index $L$ is a holomorphic function $\varphi:\Hh\times(L\otimes\C)\to\C$ satisfying
\begin{align*}
\varphi\left(\frac{a\tau+b}{c\tau+d},\frac{z}{c\tau+d}\right)
&=(c\tau+d)^k
\e\left(\frac{c(z,z)}{2(c\tau+d)}\right)\varphi(\tau,z),\\
\varphi(\tau,z+x\tau+y)
&=\e\left(-Q(x)\tau-(x,z)\right)\varphi(\tau,z)
\end{align*}
for $\begin{psmallmatrix}a&b\\c&d\end{psmallmatrix}\in\Gammaone$ and $x,y\in L$, and having a Fourier expansion
\[
\varphi(\tau,z)=\sum_{m\ge0}\sum_{\ell\in L^\dual}
c(m,\ell)q^m\e((\ell,z)).
\]
We denote the space by $J_{k,L}^{\mathrm{w}}$.  Its subspace $J_{k,L}$ of holomorphic Jacobi forms is defined by
\[
c(m,\ell)\ne0\quad\Longrightarrow\quad m\ge Q(\ell).
\]
We write $J_{*,L}^{\mathrm{w}}=\bigoplus_kJ_{k,L}^{\mathrm{w}}$ and $J_{*,L}=\bigoplus_kJ_{k,L}$.

For $\varphi\in J_{k,L}^{\mathrm{w}}$, periodicity in $z$ gives
\[
\varphi(\tau,z)=\sum_{\ell\in L^\dual}g_\ell(\tau)\e((\ell,z)).
\]
The functions $g_\ell$ are holomorphic in $\tau$, since they are obtained by integration over a compact real torus.  Comparing coefficients under $z\mapsto z+x\tau$ gives
\[
g_{\ell+x}(\tau)=q^{Q(\ell+x)-Q(\ell)}g_\ell(\tau)
\qquad(x\in L).
\]
Hence
\begin{equation}\label{eq:theta-coefficients}
f_\gamma(\tau)=q^{-Q(\ell)}g_\ell(\tau)
\quad(\ell\in\gamma+L),
\qquad
\varphi=\sum_{\gamma\in A_L}f_\gamma\theta_{L,\gamma},
\end{equation}
and $f_\gamma$ is independent of the chosen representative.

The coefficient representation is
\[
\rho_L(g)=\omega_L(g)^{-\mathsf t}
:=\bigl(\omega_L(g)^{-1}\bigr)^{\mathsf t}
\qquad(g\in\Gammaone).
\]
The transpose is taken in the basis $(\ve_\gamma)$.  This is the contragredient representation: $\rho_L(g)^{\mathsf t}\omega_L(g)=I$, so it preserves the coordinate pairing between coefficient vectors and theta vectors.  In particular,
\begin{align*}
\rho_L(T)\ve_\gamma&=\e(-Q(\gamma))\ve_\gamma,\\
\rho_L(S)\ve_\gamma&=
\frac{\e(n/8)}{\sqrt{|A_L|}}
\sum_{\delta\in A_L}\e((\gamma,\delta))\ve_\delta.
\end{align*}

For $\gamma\in A_L$, define
\[
m(\gamma)=\min_{\ell\in L+\gamma}Q(\ell),
\qquad
h(\gamma)=\{-Q(\gamma)\}.
\]

\begin{theorem}\label{thm:theta-decomposition}
The map
\[
F=\sum_{\gamma\in A_L}f_\gamma\ve_\gamma
\longmapsto
\sum_{\gamma\in A_L}f_\gamma\theta_{L,\gamma}
\]
restricts to isomorphisms
\begin{align*}
M_k(\rho_L;\cE_L^{\mathrm{w}})&\xrightarrow{\sim}J_{k+n/2,L}^{\mathrm{w}},
&
\cE_L^{\mathrm{w}}&=\{(\ve_\gamma^\dual,-m(\gamma))\}_{\gamma\in A_L},\\
M_k(\rho_L;\cE_L^{\mathrm{hol}})&\xrightarrow{\sim}J_{k+n/2,L},
&
\cE_L^{\mathrm{hol}}&=\{(\ve_\gamma^\dual,h(\gamma))\}_{\gamma\in A_L}.
\end{align*}
\end{theorem}

\begin{proof}
Equation~\eqref{eq:theta-T} follows term by term, and Poisson summation gives \eqref{eq:theta-S}.  Since $S,T$ generate $\Gammaone$, for $g=\begin{psmallmatrix}a&b\\c&d\end{psmallmatrix}$ we have
\[
\Theta_L\left(g\tau,\frac z{c\tau+d}\right)
=(c\tau+d)^{n/2}\e\left(\frac{c(z,z)}{2(c\tau+d)}\right)
\omega_L(g)\Theta_L(\tau,z).
\]
The identity $\rho_L(g)^{\mathsf t}\omega_L(g)=I$ therefore gives the modular law for $\sum_\gamma f_\gamma\theta_{L,\gamma}$.  Translation of the summation variable gives its elliptic law.

A term $q^a$ in $f_\gamma$ gives the terms
\[
q^{a+Q(\ell)}\e((\ell,z)),\qquad \ell\in\gamma+L.
\]
Since $a\equiv-Q(\gamma)\pmod{\Z}$, all resulting powers of $q$ are integers.  The weak condition is equivalent to $a\ge-m(\gamma)$, and the holomorphic condition is equivalent to $a\ge0$.  The least allowed exponent in the latter case is $h(\gamma)$.  Thus both maps take values in the stated spaces.

The elliptic Fourier supports of the $\theta_{L,\gamma}$ lie in disjoint cosets of $L$.  Hence $\sum_\gamma f_\gamma\theta_{L,\gamma}=0$ implies $f_\gamma=0$ for every $\gamma$, proving injectivity.

Conversely, take a weak or holomorphic Jacobi form of weight $k+n/2$ and extract its functions $f_\gamma$ by \eqref{eq:theta-coefficients}.  Comparing its modular law with that of $\Theta_L$, using the same independence, gives
\[
F|_kg=\rho_L(g)F,\qquad
F=\sum_\gamma f_\gamma\ve_\gamma.
\]
The Fourier conditions just proved give the corresponding bounds on $f_\gamma$.  Thus $F$ belongs to the stated source and maps to the original Jacobi form.  This proves surjectivity.
\end{proof}

For these Jacobi modules the two parity components can be treated separately.  Indeed,
\[
\rho_L(-I)\ve_\gamma=(-1)^{n/2}\ve_{-\gamma},
\]
while $m(-\gamma)=m(\gamma)$ and $h(-\gamma)=h(\gamma)$.  Replacing $\ve_\gamma,\ve_{-\gamma}$ with their sum and difference for each pair of distinct classes $\gamma,-\gamma$, and leaving self-inverse classes unchanged, therefore gives complete systems of exponent conditions on the two $(-I)$-eigenspaces, without changing the total exponent sum.  Applying Theorem~\ref{thm:basis-criterion} on these two eigenspaces gives the following criterion.

\begin{corollary}\label{cor:jacobi-basis}
Let $r=|A_L|$.  Suppose $\varphi_1,\ldots,\varphi_r$ are $\M$-linearly independent homogeneous weak Jacobi forms of weights $k_i$.  They form an $\M$-basis of $J_{*,L}^{\mathrm{w}}$ if and only if
\[
\sum_{i=1}^r k_i
=\frac{rn}{2}-12\sum_{\gamma\in A_L}m(\gamma).
\]
Likewise, $\M$-linearly independent holomorphic Jacobi forms $\psi_i$ of weights $\ell_i$ form an $\M$-basis of $J_{*,L}$ if and only if
\[
\sum_{i=1}^r\ell_i
=\frac{rn}{2}+12\sum_{\gamma\in A_L}\{-Q(\gamma)\}.
\]
\end{corollary}

\begin{remark}\label{rem:half-integral}
The freeness and weight-sum basis criterion also hold for half-integral
weights. Let $\widetilde\Gamma=\operatorname{Mp}_2(\Z)$, and let
$\nu_\eta$ be the character defined by the weight-$\tfrac12$
transformation law of $\eta$.
Suppose that $\widetilde\rho(I,-1)=-I$,
$\widetilde\rho(\widetilde T)$ has finite order, and the central lift of
$-I$ acts as a scalar, where $\widetilde T=(T,1)$.
Then $\rho=\widetilde\rho\otimes\nu_\eta^{-1}$ descends to $\Gammaone$,
since $\nu_\eta(I,-1)=-1$.
For $k\in\tfrac12+\Z$, division by $\eta$ gives an isomorphism
\[
M_k(\widetilde\rho;\cE)
\xrightarrow{\ F\mapsto F/\eta\ }
M_{k-1/2}(\rho;\cE'),
\qquad
\cE'=\{(\ell_i,c_i-\tfrac1{24})\}_{i=1}^r.
\]
Indeed, $\eta$ is nowhere zero on $\Hh$ and has leading term $q^{1/24}$.
These maps commute with multiplication by scalar modular forms.
Theorems~\ref{thm:freeness} and~\ref{thm:basis-criterion} therefore apply.
The weight-sum criterion is unchanged, because
\[
\sum_{i=1}^r\bigl(\wt(F_i)-\tfrac12\bigr)
=12\bigl(\tr\cE-\tfrac r{24}\bigr)
\quad\Longleftrightarrow\quad
\sum_{i=1}^r\wt(F_i)=12\tr\cE.
\]
In particular, Theorem~\ref{thm:invariant-theta} and
Corollary~\ref{cor:invariant-basis} extend to positive definite even
lattices of odd rank, with $k+n/2\in\Z$ in the theta correspondence.
One uses the metaplectic Weil representation and applies the above
reduction on each central eigenspace. The exponent conditions respect
these eigenspaces because the bounds at $\gamma$ and $-\gamma$ agree.
\end{remark}

\subsection{\texorpdfstring{$G$-invariant Jacobi forms}{G-invariant Jacobi forms}}

Theta decomposition is compatible with lattice automorphisms.
We use this compatibility to identify invariant Jacobi modules with
modules of vector-valued modular forms and deduce freeness and the basis criteria.

Let $G\subset\OO(L)$.  Its action on $\C[A_L]$ is
\[
g\cdot\ve_\gamma=\ve_{g\gamma}\qquad(g\in G,\ \gamma\in A_L).
\]
Since $G$ preserves $Q$ and the discriminant pairing, this action commutes with $\rho_L$.  Hence $\cV_L^G:=\C[A_L]^G$ is a $\Gammaone$-subrepresentation.

The action on Jacobi forms is
\[
(g\cdot\varphi)(\tau,z)=\varphi(\tau,g^{-1}z).
\]
We denote its fixed spaces by $J_{k,L}^{\mathrm{w},G}$ and $J_{k,L}^{G}$, and write
\[
J_{*,L}^{\mathrm{w},G}=\bigoplus_kJ_{k,L}^{\mathrm{w},G},
\qquad
J_{*,L}^{G}=\bigoplus_kJ_{k,L}^{G}.
\]
Theta decomposition is equivariant for these actions, since
$\theta_{L,\gamma}(\tau,g^{-1}z)=\theta_{L,g\gamma}(\tau,z)$.

For a $G$-orbit $\cO\subset A_L$, define
\[
\ve_{\cO}=\sum_{\gamma\in\cO}\ve_\gamma,\qquad
-\cO=\{-\gamma:\gamma\in\cO\}.
\]
Negation commutes with $G$, so $-\cO$ is again a $G$-orbit.  The orbit sums form a basis of $\cV_L^G$.  Its dual basis is denoted by $\ve_{\cO}^\dual$, with
$\ve_{\cO}^\dual(\ve_{\cO'})=\delta_{\cO,\cO'}$.

Let $\rho_L^G$ be the restriction of $\rho_L$ to $\cV_L^G$.  For an orbit $\cO$ and any $\gamma\in\cO$, define
\[
m(\cO)=\min_{x\in\gamma+L}Q_L(x),\qquad
Q(\cO)=Q(\gamma)\in\Q/\Z.
\]
These definitions are independent of $\gamma$ since $G$ acts by lattice isometries.

\begin{theorem}\label{thm:invariant-theta}
Theta decomposition restricts to isomorphisms
\begin{align*}
J_{k+n/2,L}^{\mathrm{w},G}
&\cong M_k(\rho_L^G;\cE_{L,G}^{\mathrm{w}}),\\
J_{k+n/2,L}^{G}
&\cong M_k(\rho_L^G;\cE_{L,G}^{\mathrm{hol}}),
\end{align*}
where
\begin{align*}
\cE_{L,G}^{\mathrm{w}}
&=\{(\ve_{\cO}^\dual,-m(\cO)):\cO\in G\backslash A_L\},\\
\cE_{L,G}^{\mathrm{hol}}
&=\{(\ve_{\cO}^\dual,\{-Q(\cO)\}):\cO\in G\backslash A_L\}.
\end{align*}
In particular, both invariant Jacobi modules are free over $\M$ of rank $|G\backslash A_L|$.
\end{theorem}

\begin{proof}
A Jacobi form is $G$-invariant if and only if its theta coefficients are constant on $G$-orbits.  Thus its coefficient vector takes values in $\cV_L^G$.  Since the $G$-action commutes with $\rho_L$, this is a subrepresentation, and Theorem~\ref{thm:theta-decomposition} gives the asserted exponent conditions.  To apply Theorem~\ref{thm:freeness}, pair the orbits $\cO$ and $-\cO$.  They have the same minimum and quadratic value.  Their sum and difference are eigenvectors of $\rho_L(-I)$ with the same prescribed lower bound, while an orbit fixed by negation already has scalar $(-I)$-action.  The two resulting complete systems describe the even- and odd-weight modules separately.  Theorem~\ref{thm:freeness} applied to these two modules gives total rank $|G\backslash A_L|$.
\end{proof}

\begin{corollary}\label{cor:invariant-basis}
Put $r=|G\backslash A_L|$.  $\M$-linearly independent weak $G$-invariant Jacobi forms of weights $k_1,\ldots,k_r$ form an $\M$-basis if and only if
\[
\sum_i k_i
=\frac{rn}{2}-12\sum_{\cO\in G\backslash A_L}m(\cO).
\]
$\M$-linearly independent holomorphic $G$-invariant Jacobi forms of weights $\ell_i$ form an $\M$-basis if and only if
\[
\sum_i\ell_i
=\frac{rn}{2}+12\sum_{\cO\in G\backslash A_L}\{-Q(\cO)\}.
\]
\end{corollary}

The following principal-part criterion compares prescribed growth
with the standard holomorphic condition. It will be used to pass
from weak to holomorphic Jacobi forms in Section~\ref{sec:Dfamily}.

\begin{proposition}\label{prop:principal-parts}
Let $\rho$ have finite image and $\rho(-I)=I$.  Choose a $T$-eigenbasis
$v_1,\ldots,v_r$, and let
\[
\cE=\{(v_i^\dual,-\mu_i):1\le i\le r\},
\qquad \mu_i\ge0.
\]
Let $\mathcal P(\cE)$ be the space of allowed principal parts
\[
P=\sum_{i=1}^r\sum_{0\le j<\lceil\mu_i\rceil}
c_{i,j}q^{-\mu_i+j}v_i.
\]
For every even $k$, there is a sequence, exact through $\mathcal P(\cE)$,
\begin{equation}\label{eq:principal-part-sequence}
0\longrightarrow M_k(\rho)\longrightarrow M_k(\rho;\cE)
\stackrel{\operatorname{pp}}{\longrightarrow}\mathcal P(\cE)
\stackrel{\operatorname{ob}_k}{\longrightarrow}
S_{2-k}(\rho^\dual)^\dual.
\end{equation}
If
$g=\sum_i\sum_{\nu>0}a_i(\nu)q^\nu v_i^\dual
\in S_{2-k}(\rho^\dual)$, the last map is
\[
\operatorname{ob}_k(P)(g)
=\sum_i\sum_{0\le j<\lceil\mu_i\rceil}
c_{i,j}a_i(\mu_i-j).
\]
In particular, if $S_{2-k}(\rho^\dual)=0$, then
\begin{equation}\label{eq:principal-part-difference}
\dim M_k(\rho;\cE)-\dim M_k(\rho)
=\sum_i\lceil\mu_i\rceil.
\end{equation}
This applies to every even $k\ge2$.
\end{proposition}

\begin{proof}
Take $F\in M_k(\rho;\cE)$ and $g\in S_{2-k}(\rho^\dual)$.
The scalar pairing $g(F)$ is weakly holomorphic of weight $2$.
Its constant term is zero by the residue theorem, so
\[
\operatorname{ob}_k(\operatorname{pp}F)(g)
=[q^0]\,g(F)=0.
\]
Hence $\operatorname{ob}_k\circ\operatorname{pp}=0$.

Conversely, Borcherds' obstruction theorem \cite[Theorem~3.1]{Bo99} gives a criterion for prescribed nonpositive Fourier terms.
Such terms are realized by a weakly holomorphic form exactly when their constant-term pairings with $M_{2-k}(\rho^\dual)$ vanish.
For $P\in\ker\operatorname{ob}_k$, the functional $g\mapsto[q^0]\,g(P)$ on $M_{2-k}(\rho^\dual)$ vanishes on $S_{2-k}(\rho^\dual)$.
It therefore factors through the constant-term map and can be canceled by choosing a suitable constant term for the desired form.
The cited theorem then gives a weakly holomorphic form $F$ with negative Fourier part $P$.
Since $P\in\mathcal P(\cE)$, this form satisfies the bounds $\cE$.
Thus
\[
\ker\operatorname{ob}_k=\operatorname{im}\operatorname{pp}.
\]
For negative weight and Weil representations, this principal-part criterion is also given in \cite[Theorem~1.17]{Br02}.

The kernel of $\operatorname{pp}$ consists of the holomorphic forms.
The allowed negative exponents in the $i$-th component are
$-\mu_i,-\mu_i+1,\ldots,-\mu_i+\lceil\mu_i\rceil-1$,
so $\dim\mathcal P(\cE)=\sum_i\lceil\mu_i\rceil$.
Finally, a finite-image representation has no cusp forms of
nonpositive weight: on a finite-index kernel its components are
scalar holomorphic forms, which vanish in negative weight, and
are constant in weight zero.  This proves the last assertion.
\end{proof}

\section{Weyl invariant Jacobi forms of lattice index \texorpdfstring{$D_{8d}^{+}(2)$}{D8d+(2)}}\label{sec:Dfamily}

We determine the generator weights of the weak and holomorphic Weyl
invariant Jacobi modules of lattice index $D_{8d}^{+}(2)$ for $d\ge2$.
After computing the Weyl-group orbits on the discriminant group, we
construct a weak Jacobi form from a form on $D_8$. Its theta coefficient
vector and successive derivatives generate a submodule corresponding
to a representation of dimension $2d+1$. We prove that projection to
the remaining three-dimensional representation is surjective with
the induced cusp conditions, which determines the weak generator weights.
For holomorphic forms, we compare principal parts in positive
vector-valued weight and determine the weight-zero dimension by matching
upper and lower bounds. These dimensions give the holomorphic generator
weights and their multiplicities.

\subsection{The lattice and its discriminant representation}

We determine the invariant Weil representation for $D_{8d}^{+}(2)$
using the orbits of $W(D_{8d})$ on its discriminant group.

Fix $d\ge2$, and let $e_1,\ldots,e_{8d}$ be the standard coordinate vectors of $\mathbb R^{8d}$.  Write
\[
D_{8d}=\{x\in\Z^{8d}:\textstyle\sum_i x_i\equiv0\pmod2\},
\qquad h=\tfrac12(1,\ldots,1),
\qquad D_{8d}^{+}=D_{8d}\cup(h+D_{8d}).
\]
The lattice $D_{8d}^{+}$ is even unimodular, with $Q_{D_{8d}^{+}}(h)=d$.  Put $L_d=D_{8d}^{+}(2)$ and $G_d=\OO(D_{8d}^{+})$.

The Weyl group $W(D_{8d})$ consists of coordinate permutations and sign changes in an even number of coordinates.  Its elements act by
\[
g_{\pi,\varepsilon}(x)_i=\varepsilon_i x_{\pi^{-1}(i)},
\qquad
\pi\in S_{8d},\quad \varepsilon_i\in\{1,-1\},\quad\prod_i\varepsilon_i=1.
\]
Every vector in $h+D_{8d}$ has squared norm at least $8d/4=2d\ge4$.
Thus the norm-two roots of $D_{8d}^{+}$ are exactly those of $D_{8d}$,
and $W(D_{8d}^{+})=W(D_{8d})$.
Every lattice automorphism therefore preserves this root system and is a
signed coordinate permutation.
An odd number of sign changes does not preserve $h+D_{8d}$,
whereas an even number does. Hence
\begin{equation}\label{eq:Gd-Weyl}
G_d=\OO(D_{8d}^{+})=W(D_{8d}^{+})=W(D_{8d}).
\end{equation}

The group underlying $L_d$ is $D_{8d}^{+}$, and its bilinear form is twice that of $D_{8d}^{+}$.  Unimodularity gives
\[
A_d:=L_d^\dual/L_d\cong D_{8d}^{+}/2D_{8d}^{+},
\qquad
x\bmod2D_{8d}^{+}\longmapsto x/2+D_{8d}^{+}.
\]
For $x,y\in D_{8d}^{+}$ we have
\begin{equation}\label{eq:Q-identification}
Q_{L_d}(x/2)\equiv\tfrac12Q_{D_{8d}^{+}}(x),
\qquad
(x/2,y/2)_{L_d}\equiv\tfrac12(x,y)_{D_{8d}^{+}}\pmod{\Z}.
\end{equation}
The action on $A_d$ is $g(x/2+D_{8d}^{+})=gx/2+D_{8d}^{+}$.  We call a class \emph{integral} if its representative $x$ lies in $D_{8d}$, and a \emph{spin class} if $x$ lies in $h+D_{8d}$.  These two parts are $G_d$-invariant.

In the following representatives, $a^r$ denotes $r$ repeated coordinates equal to $a$.  The representatives lie in $D_{8d}^{+}/2D_{8d}^{+}$; dividing them by $2$ gives their classes in $A_d$.

\begin{proposition}\label{prop:orbits}
The action of $G_d$ on $A_d$ has the following $2d+4$ orbits.
\begin{enumerate}[label=\textup{(\roman*)}]
\item The integral orbit $\cO_m$, $0\le m\le2d$, has representative
\[
x_m=(0^{8d-2m},1^{2m}),\qquad Q_{D_{8d}^{+}}(x_m)=m,
\]
and cardinality
\[
|\cO_m|=
\begin{cases}
1,&m=0,\\[1mm]
2\binom{8d}{2m},&1\le m<2d,\\[1mm]
\binom{8d}{4d},&m=2d.
\end{cases}
\]
\item The integral orbit $\cO_{\mathrm v}$ has representative $(2,0^{8d-1})$, quadratic value $Q_{D_{8d}^{+}}=2$, and cardinality $1$.
\item The spin orbit $\cO_{\mathrm s}$ has representative $(\tfrac12)^{8d}$, quadratic value $Q_{D_{8d}^{+}}=d$, and cardinality $2^{8d-2}$.
\item The spin orbit $\cO_{\mathrm c}$ has representative $(-\tfrac32,(\tfrac12)^{8d-1})$, quadratic value $Q_{D_{8d}^{+}}=d+1$, and cardinality $2^{8d-2}$.
\end{enumerate}
In particular, $\dim\C[A_d]^{G_d}=2d+4$.
\end{proposition}

\begin{proof}
We have $2D_{8d}^{+}=2D_{8d}\cup(2h+2D_{8d})$.
Let $x\in D_{8d}$ and write $x=p+2z$, where $p_i\in\{0,1\}$.
The number of entries equal to $1$ in $p$ is even.
For fixed $p$, the class modulo $2D_{8d}$ depends only on
$\sum_i z_i\pmod2$.
Subtracting $2h$ replaces $p$ by its complement $(1,\ldots,1)-p$ and preserves this parity.
Choose the complement, if necessary, so that $p$ has $2m\le4d$ entries equal to $1$.
After permuting coordinates, the class has one of the representatives
\[
(0^{8d-2m},1^{2m}),\qquad
(2,0^{8d-2m-1},1^{2m}).
\]
For $m=0$, these give $\cO_0$ and $\cO_{\mathrm v}$.
For $m>0$, change the signs of the first coordinate and one coordinate equal to $1$ in the second representative.
The result differs from the first representative by an element of $2D_{8d}$.
Thus both classes belong to $\cO_m$.

Signed permutations preserve the number of odd coordinates.
Subtracting $2h$ replaces a count $2m$ by $8d-2m$.
Hence different values of $m$ give different orbits.
For $1\le m<2d$, there are $\binom{8d}{2m}$ choices of the odd coordinates and two classes for each choice.
For $m=2d$, complementation pairs the choices of $4d$ coordinates.
Each pair gives two classes, so $|\cO_{2d}|=\binom{8d}{4d}$.
For $m=0$, the two classes are $0$ and $2e_1$.
They are distinct, and every signed coordinate vector $\pm2e_i$ is congruent to $2e_1$ modulo $2D_{8d}$.
Both orbits are therefore singletons.

Now let $x=h+y$ with $y\in D_{8d}$.
Change the signs of the coordinates for which $y_i$ is odd.
There are an even number of such coordinates.
The resulting vector has the form $h+2z$.
Its class modulo $2D_{8d}$ is $h$ or $h+2e_1$, according to the parity of $\sum_i z_i$.
Since $4e_1\in2D_{8d}$, the latter class also equals $h-2e_1$.
Thus every spin class belongs to $\cO_{\mathrm s}$ or $\cO_{\mathrm c}$.

The orbit of $h$ is represented by
$(\varepsilon_1/2,\ldots,\varepsilon_{8d}/2)$ with $\prod_i\varepsilon_i=1$.
Two such vectors represent the same class exactly when they are equal or opposite.
Indeed, their difference has entries in $\{0,1,-1\}$, whereas an element of $2D_{8d}^{+}$ has all coordinates even or all coordinates odd.
There are therefore $2^{8d-2}$ classes in $\cO_{\mathrm s}$.
Translation by the fixed class $\cO_{\mathrm v}$ sends this orbit to $\cO_{\mathrm c}$.
The two spin orbits have the same size.
Their quadratic values differ by $1/2$ modulo $\Z$, by \eqref{eq:Q-identification}, so they are distinct.
They are also distinct from the integral orbits.
Finally, squaring the displayed representatives gives the stated quadratic values.
\end{proof}

By Theorem~\ref{thm:invariant-theta}, weak $G_d$-invariant Jacobi forms correspond to vector-valued modular forms for the representation
\[
\rho_d:=\rho_{L_d}^{G_d}
\quad\text{on}\quad
\cV_d:=\C[A_d]^{G_d}.
\]
We now determine the Fourier bounds on their theta coefficients.
For each orbit $\cO$, the covector $\ve_{\cO}^\dual\in\cV_d^\dual$ extracts its orbit-sum coefficient:
\[
\ve_{\cO}^\dual\!\left(\sum_{\cO'}f_{\cO'}\ve_{\cO'}\right)=f_{\cO}.
\]
The representatives in Proposition~\ref{prop:orbits} have minimum norm in their classes.
For an integral orbit $\cO_m$, choose the parity pattern with $2m\le4d$ odd coordinates, using complementation when necessary.
Each odd coordinate has absolute value at least one, and $x_m$ attains the resulting norm bound.
The nonzero class of even vectors has shortest representative $2e_1$.
For spin classes, all coordinates have absolute value at least $1/2$.
Every representative of a class in $\cO_{\mathrm c}$ has at least one coordinate of absolute value at least $3/2$.
The displayed representative attains the lower bound.
Hence the minima in the scaled lattice are
\[
m(\cO_m)=\frac m2,\qquad m(\cO_{\mathrm v})=1,\qquad
m(\cO_{\mathrm s})=\frac d2,\qquad m(\cO_{\mathrm c})=\frac{d+1}{2}.
\]
Define the complete system of exponent conditions on $\cV_d$ by
\begin{equation}\label{eq:Ed-weak}
\begin{aligned}
\cE_d^{\mathrm{w}}={}&\{(\ve_{\cO_m}^\dual,-m/2):0\le m\le2d\}\\
&\cup\{(\ve_{\cO_{\mathrm v}}^\dual,-1),
(\ve_{\cO_{\mathrm s}}^\dual,-d/2),
(\ve_{\cO_{\mathrm c}}^\dual,-(d+1)/2)\}.
\end{aligned}
\end{equation}
Thus a form
\[
F=f_0\ve_{\cO_0}+\sum_{m=1}^{2d}f_m\ve_{\cO_m}
 +f_{\mathrm v}\ve_{\cO_{\mathrm v}}
 +f_{\mathrm s}\ve_{\cO_{\mathrm s}}
 +f_{\mathrm c}\ve_{\cO_{\mathrm c}}
\]
belongs to $M_k(\rho_d;\cE_d^{\mathrm{w}})$ exactly when
\[
f_m=O(q^{-m/2})\ (0\le m\le2d),\quad
f_{\mathrm v}=O(q^{-1}),\quad
f_{\mathrm s}=O(q^{-d/2}),\quad
f_{\mathrm c}=O(q^{-(d+1)/2}).
\]
Theta decomposition identifies this space with $J_{k+4d,L_d}^{\mathrm{w},G_d}$.

Since $Q_{L_d}(A_d)\subset\frac12\Z/\Z$, we have $\rho_d(T)^2=I$.  Also $A_d$ is $2$-elementary, so the Weil formula gives $\rho_d(S)^2=I$.  Thus $\rho_d$ factors through
\[
\langle S,T:S^2=T^2=(ST)^3=1\rangle\cong S_3.
\]
Let $\mathbf1$ denote the trivial representation of $S_3$.  Identify the images of $S$ and $T$ with $(12)$ and $(23)$, respectively.  On
\[
\{(z_1,z_2,z_3)\in\C^3:z_1+z_2+z_3=0\},
\]
the standard representation $\sigma$ is given by
\[
\sigma(S)(z_1,z_2,z_3)=(z_2,z_1,z_3),
\qquad
\sigma(T)(z_1,z_2,z_3)=(z_1,z_3,z_2).
\]
These formulas define an irreducible two-dimensional representation.  We view $\mathbf1$ and $\sigma$ as representations of $\Gammaone$ through this quotient.

Set
\[
U_d=
\Span\{\ve_{\cO_0}-\ve_{\cO_{\mathrm v}},
\ve_{\cO_{\mathrm s}},
\ve_{\cO_{\mathrm c}}\}.
\]
For a convenient normalization put
\[
u_0=2^{4d-2}(\ve_{\cO_0}-\ve_{\cO_{\mathrm v}}).
\]
Choose $u_1$ from $\ve_{\cO_{\mathrm s}},\ve_{\cO_{\mathrm c}}$ so that $Tu_1=u_1$, and let $u_2$ be the other vector.  Explicitly,
\[
(u_1,u_2)=
\begin{cases}
(\ve_{\cO_{\mathrm s}},\ve_{\cO_{\mathrm c}}),&d\text{ even},\\
(\ve_{\cO_{\mathrm c}},\ve_{\cO_{\mathrm s}}),&d\text{ odd}.
\end{cases}
\]

\begin{lemma}\label{lem:U-matrices}
The subspace $U_d$ is invariant under $\rho_d$.  On the basis $u_0,u_1,u_2$ the generators act by
\[
T:u_0\mapsto u_0,
\quad u_1\mapsto u_1,
\quad u_2\mapsto-u_2,
\]
and
\begin{align*}
S u_0&=\frac12u_1+\frac12u_2,\\
S u_1&=u_0+\frac12u_1-\frac12u_2,\\
S u_2&=u_0-\frac12u_1+\frac12u_2.
\end{align*}
Consequently,
\[
U_d=U_d^{(1)}\oplus U_d^{(2)},
\]
where
\[
U_d^{(1)}=\C(u_0+u_1)
\]
is the trivial representation and
\[
U_d^{(2)}=\Span\{2u_0-u_1,u_2\}
\]
is the irreducible two-dimensional standard representation of $S_3$.
\end{lemma}

\begin{proof}
The action of $T$ follows from \eqref{eq:Q-identification}.
We compute the action of $S$ by finite Fourier sums.
Put $v=(2,0^{8d-1})$.
The character
\[
x/2+D_{8d}^{+}\longmapsto\e((v/2,x/2)_{L_d})
\]
is $1$ on integral classes and $-1$ on spin classes.
Hence
\[
S(\ve_{\cO_0}-\ve_{\cO_{\mathrm v}})
=2^{1-4d}(\ve_{\cO_{\mathrm s}}+\ve_{\cO_{\mathrm c}}),
\qquad
Su_0=\tfrac12(\ve_{\cO_{\mathrm s}}+\ve_{\cO_{\mathrm c}}).
\]

Let $y\in D_{8d}$ represent an integral class.
Each class of $\cO_{\mathrm s}$ is represented by two opposite sign vectors.
The coefficient sum over $\cO_{\mathrm s}$ is therefore
\begin{align*}
&\frac12\sum_{\substack{\varepsilon_a=\pm1\\\prod_a\varepsilon_a=1}}
 \exp\!\left(\frac{\pi i}{2}\sum_a\varepsilon_a y_a\right)\\
&\qquad=\frac14\left\{
 \prod_{a=1}^{8d}2\cos(\pi y_a/2)
 +\prod_{a=1}^{8d}2i\sin(\pi y_a/2)\right\}.
\end{align*}
Here we used the indicator $(1+\prod_a\varepsilon_a)/2$ to remove the restriction on the signs.
For $y=x_m$ with $1\le m\le2d$, one coordinate is odd and another is zero.
Thus both products vanish.
At $y=0$ and $y=v$, the sums are $2^{8d-2}$ and $-2^{8d-2}$.
After division by the Weil normalization $2^{4d}$, the integral part of $S\ve_{\cO_{\mathrm s}}$ is $u_0$.

For the spin coefficients, fix
$y=(\delta_1/2,\ldots,\delta_{8d}/2)$ with $\prod_i\delta_i=1$.
The same indicator gives
\begin{align*}
\sum_{\gamma\in\cO_{\mathrm s}}\e((\gamma,y/2+D_{8d}^{+})_{L_d})
&=\frac12\sum_{\substack{\varepsilon_i=\pm1\\\prod_i\varepsilon_i=1}}
 \exp\!\left(\frac{\pi i}{4}\sum_i\varepsilon_i\delta_i\right)\\
&=\frac14\left\{\prod_i\left(e^{\pi i\delta_i/4}+e^{-\pi i\delta_i/4}\right)
+\prod_i\left(e^{\pi i\delta_i/4}-e^{-\pi i\delta_i/4}\right)\right\}\\
&=2^{4d-1}.
\end{align*}
The last equality uses $\prod_i\delta_i=1$ and $i^{8d}=1$.
Since $\cO_{\mathrm c}=\cO_{\mathrm s}+\cO_{\mathrm v}$, the sum at a class in $\cO_{\mathrm c}$ has the opposite sign.
Dividing by $2^{4d}$ gives
\[
S\ve_{\cO_{\mathrm s}}=u_0+\tfrac12\ve_{\cO_{\mathrm s}}-\tfrac12\ve_{\cO_{\mathrm c}},
\qquad
S\ve_{\cO_{\mathrm c}}=u_0-\tfrac12\ve_{\cO_{\mathrm s}}+\tfrac12\ve_{\cO_{\mathrm c}}.
\]
For the second identity, translation by $\cO_{\mathrm v}$ changes the signs of the spin coefficients and leaves the integral coefficients unchanged.
These formulas give the stated action on $u_0,u_1,u_2$ for both parities of $d$.

The line $\C(u_0+u_1)$ is fixed by $S$ and $T$.
The span of $2u_0-u_1,u_2$ is invariant, with
\begin{align*}
T(2u_0-u_1)&=2u_0-u_1,& Tu_2&=-u_2,\\
S(2u_0-u_1)&=-\tfrac12(2u_0-u_1)+\tfrac32u_2,&
Su_2&=\tfrac12(2u_0-u_1)+\tfrac12u_2.
\end{align*}
Neither $T$-eigenline is preserved by $S$.
This representation is irreducible, so it is the standard representation $\sigma$ of $S_3$.
\end{proof}

\subsection{The weak case}

We determine the generator weights of the weak Jacobi module by constructing
a cyclic summand of dimension $2d+1$ from a weak Jacobi form on $D_8$.
We prove that projection to the remaining three-dimensional representation
is surjective with the induced cusp conditions, and then apply
Theorem~\ref{thm:rank3-classification}.

For an even unimodular lattice $K$, write
\[
\Theta_K(\tau,z)=\sum_{x\in K}q^{Q_K(x)}\e((x,z)_K).
\]
We identify $D_8$ with its usual sublattice in $E_8=D_8^+$ and with the sublattice in the first eight coordinates of $D_{16}^+$.  The notation $\Theta_K|_{D_8}$ means restriction of the elliptic variable along these inclusions.  Here $W(D_8)$ has the same signed-permutation description in eight coordinates.  We use the $W(D_8)$-invariant weak Jacobi form
\begin{equation}\label{eq:D8-seed}
\widetilde\varphi_{-4,1}^{D_8}
=
\frac{E_4\,\Theta_{E_8}|_{D_8}-\Theta_{D_{16}^+}|_{D_8}}{\Delta}
\in J_{-4,D_8}^{\mathrm{w},W(D_8)}.
\end{equation}
This is the form denoted with a tilde in \cite[Section~3.1.2]{AG20}; it is $W(D_8)$-invariant but not $\OO(D_8)$-invariant.  The Weyl invariance is sufficient for the construction below.

There is an explicit isometric embedding
\[
\jmath_d:L_d\hookrightarrow D_8^{\oplus d}.
\]
On each block of eight coordinates it is given by
\begin{equation}\label{eq:jmath}
(x_1,\ldots,x_8)
\longmapsto
(x_1+x_2,x_1-x_2,\ldots,x_7+x_8,x_7-x_8).
\end{equation}
The block matrix $B$ in \eqref{eq:jmath} satisfies $B^{\mathsf t}B=2I_8$.  In either coset $D_{8d}$ or $h+D_{8d}$, the output coordinates are integral, and their sum in each block is $2(x_1+x_3+x_5+x_7)\in2\Z$.  Thus the image lies in $D_8^{\oplus d}$, and the map is isometric for the scaled form on $L_d$.

For a Jacobi form on $D_8^{\oplus d}$, its pullback is
\[
(\jmath_d^*\varphi)(\tau,z)=\varphi(\tau,\jmath_d z),
\]
where $\jmath_d$ is extended linearly to the complexified lattices.  For $z^{(1)},\ldots,z^{(d)}\in D_8\otimes\C$, define the product
\[
\varphi_d(\tau,z^{(1)},\ldots,z^{(d)})
=\prod_{a=1}^d\widetilde\varphi_{-4,1}^{D_8}(\tau,z^{(a)})
\in J_{-4d,D_8^{\oplus d}}^{\mathrm{w}}.
\]
To obtain a $G_d$-invariant form, sum the pullback over $G_d$:
\begin{equation}\label{eq:Psi-d}
\Psi_d=\sum_{g\in G_d}g\cdot\jmath_d^*\varphi_d
\in J_{-4d,L_d}^{\mathrm{w},G_d}.
\end{equation}
Let
\[
\Phi_d\in M_{-8d}(\rho_d;\cE_d^{\mathrm{w}})
\]
be the theta coefficient vector of $\Psi_d$ in \eqref{eq:Psi-d}.

\begin{proposition}\label{prop:Phi-leading}
There are nonzero constants $a_0,\ldots,a_{2d}$ such that
\begin{align}\label{eq:Phi-leading}
\Phi_d
&=(a_0+O(q))(\ve_{\cO_0}+\ve_{\cO_{\mathrm v}})\\
&\quad+
\sum_{m=1}^{2d}
\left(a_mq^{-m/2}+O(q^{-m/2+1})\right)\ve_{\cO_m}.
\nonumber
\end{align}
In particular,
\[
V_d^{\mathrm{cyc}}
:=V(\Phi_d)
=
\Span\{\ve_{\cO_0}+\ve_{\cO_{\mathrm v}},
\ve_{\cO_1},\ldots,\ve_{\cO_{2d}}\}
\]
has dimension $2d+1$.
\end{proposition}

\begin{proof}
Write $f_m,f_{\mathrm v},f_{\mathrm s},f_{\mathrm c}$ for the orbit coefficients of $\Phi_d$.
First we compute the $q^0$-term of $\Psi_d$ and read off the leading terms of $f_m$.
Next we prove that $f_{\mathrm s}=f_{\mathrm c}=0$.
Finally we show that $f_0=f_{\mathrm v}$.

For $b\in D_8^\dual$, the pullback of a Fourier character on one block is
\[
\e((b,Bz)_{D_8})
=\e((B^{\mathsf t}b,z)_{D_8})
=\e((B^{\mathsf t}b/2,z)_{D_8(2)}).
\]
Thus a product of characters from the $d$ blocks contributes to the theta coefficient of the class
\[
(B^{\mathsf t}b^{(1)},\ldots,B^{\mathsf t}b^{(d)})+2D_{8d}^{+}
\quad\text{in } A_d\cong D_{8d}^{+}/2D_{8d}^{+}.
\]
Here $b^{(a)}\in D_8^\dual$ is the exponent vector in block $a$.

Write the coordinates of $z$ as $z_{a,i}$, with $1\le a\le d$ and $1\le i\le8$, and put
$Z_{a,i}=\e(z_{a,i})+\e(-z_{a,i})$.
Substituting \eqref{eq:jmath} into the $q^0$-term of \eqref{eq:D8-seed}, given in \cite[Section~3.1.2]{AG20}, yields
\[
\begin{aligned}
[q^0](\jmath_d^*\varphi_d)
=\prod_{a=1}^d\Biggl(&128-16\sum_{i=1}^4 Z_{a,2i-1}Z_{a,2i}\\
&+\sum_{\substack{I\subset\{1,2,3,4\}\\ |I|\text{ even}}}
\prod_{i\in I}Z_{a,2i}\prod_{i\notin I}Z_{a,2i-1}\Biggr).
\end{aligned}
\]
In the last sum, each pair $(2i-1,2i)$ contributes one factor: $Z_{a,2i}$ if $i\in I$, and $Z_{a,2i-1}$ if $i\notin I$.
For example,
\[
I=\varnothing:\quad Z_{a,1}Z_{a,3}Z_{a,5}Z_{a,7},
\qquad
I=\{1,2\}:\quad Z_{a,2}Z_{a,4}Z_{a,5}Z_{a,7}.
\]
The allowed subsets $I$ are $\varnothing$, the six subsets with two elements, and $\{1,2,3,4\}$.
Each block therefore contributes $0$, $2$, or $4$ nonzero coordinates, with coefficient signs $+$, $-$, or $+$, respectively.
Every monomial in the product is a character
\[
\e((x,z)_{D_{8d}^{+}}),\qquad
x\in\{0,\pm1\}^{8d},\qquad
\#\{i:x_i\ne0\}=2m\quad(0\le m\le2d).
\]
Its coefficient has sign $(-1)^m$, and every $0\le m\le2d$ occurs.
The group $G_d$ acts transitively on the vectors in $\{0,\pm1\}^{8d}$ with $2m$ nonzero coordinates.
Indeed, there is always a zero coordinate, so every sign change can be paired with a sign change on a zero coordinate.
Hence averaging gives
\[
[q^0]\Psi_d
=a_0+\sum_{m=1}^{2d}a_m
\sum_{\substack{x\in\{0,\pm1\}^{8d}\\ \#\{i:x_i\ne0\}=2m}}
\e((x,z)_{D_{8d}^{+}}),
\qquad (-1)^m a_m>0.
\]
The character $\e((x_m,z)_{D_{8d}^{+}})$ belongs to the class $x_m+2D_{8d}^{+}$ in $\cO_m$.
Its exponent in the corresponding theta series is $Q_{L_d}(x_m/2)=m/2$.
Thus \eqref{eq:theta-coefficients} gives
\[
q^{m/2}f_m=a_m+O(q)\qquad(0\le m\le2d).
\]

Now consider all Fourier terms.
For every $b\in D_8^\dual$, its coordinates are all integers or all half-integers.
The coordinates $b_{2i-1}\pm b_{2i}$ of $B^{\mathsf t}b$ are therefore integers, and their sum $2(b_1+b_3+b_5+b_7)$ is even.
Thus $B^{\mathsf t}b\in D_8$.
The class displayed above is consequently represented by an element of $D_{8d}$, so it is integral.
Since $G_d$ preserves integral classes, no spin class occurs in $\Psi_d$.
Hence $f_{\mathrm s}=f_{\mathrm c}=0$.

Finally, take the $\cO_{\mathrm s}$-coefficient in the $S$-transformation law for $\Phi_d$.
The orbit sums are orthogonal, so Lemma~\ref{lem:U-matrices} and the unitarity of $\rho_d$ give
\[
0=\tau^{8d}f_{\mathrm s}(-1/\tau)
=2^{-4d}(f_0-f_{\mathrm v}).
\]
Hence $f_0=f_{\mathrm v}$.

The functions $f_0,f_1,\ldots,f_{2d}$ have distinct lowest Fourier exponents $0,-1/2,\ldots,-d$.
They are linearly independent over $\C$.
Lemma~\ref{lem:value-span} gives the asserted value span and dimension.
\end{proof}

On $V_d^{\mathrm{cyc}}$, put
\[
v_0=\ve_{\cO_0}+\ve_{\cO_{\mathrm v}},\qquad
v_m=\ve_{\cO_m}\quad(1\le m\le2d),
\]
and let $v_0^\dual,\ldots,v_{2d}^\dual$ be the dual basis, so that
$v_m^\dual(v_l)=\delta_{ml}$.  Write
$\rho_d^{\mathrm{cyc}}=\rho_d|_{V_d^{\mathrm{cyc}}}$ and define
\begin{equation}\label{eq:Ecyc}
\cE_d^{\mathrm{cyc}}=\{(v_m^\dual,-m/2):0\le m\le2d\}.
\end{equation}
For a form $F=\sum_{m=0}^{2d}f_m v_m$, the $\cO_0$- and $\cO_{\mathrm v}$-coefficients are both $f_0$.  Their two conditions in \eqref{eq:Ed-weak} are $f_0=O(1)$ and $f_0=O(q^{-1})$.  The reduction in \eqref{eq:reduced-exponent-conditions} retains the stronger condition $f_0=O(1)$.  The $\cO_{\mathrm s}$- and $\cO_{\mathrm c}$-coefficients are zero, and the remaining conditions are $f_m=O(q^{-m/2})$.  Thus \eqref{eq:Ecyc} describes exactly the original conditions on forms taking values in $V_d^{\mathrm{cyc}}$.  In particular,
\[
\tr\cE_d^{\mathrm{cyc}}
=-\frac12\sum_{m=1}^{2d}m
=-\frac{d(2d+1)}2.
\]

\begin{corollary}\label{cor:cyclic-basis}
The forms $\mathcal D^j\Phi_d$, $0\le j\le2d$, form a homogeneous
basis of $M(\rho_d^{\mathrm{cyc}};\cE_d^{\mathrm{cyc}})$,
with weights $-8d+2j$.
\end{corollary}

\begin{proof}
Proposition~\ref{prop:derivative-independent} shows that these forms are independent.
Their weights sum to
\[
\sum_{j=0}^{2d}(-8d+2j)=-6d(2d+1)=12\tr\cE_d^{\mathrm{cyc}}.
\]
Theorem~\ref{thm:basis-criterion} therefore makes them a basis.
\end{proof}

By Proposition~\ref{prop:Phi-leading}, $V_d^{\mathrm{cyc}}$ is an invariant subspace of dimension $2d+1$.  Lemma~\ref{lem:U-matrices} gives an invariant subspace $U_d$ of dimension $3$.  Their displayed bases together form a basis of $\cV_d$, which has dimension $2d+4$.  Hence
\begin{equation}\label{eq:representation-splitting}
\cV_d=V_d^{\mathrm{cyc}}\oplus U_d
\end{equation}
is a decomposition into subrepresentations.

Let $P_U:\cV_d\to U_d$ be the projection associated with \eqref{eq:representation-splitting}, and put $\rho_d^U=\rho_d|_{U_d}$.  Let
\[
u_0^\dual,u_1^\dual,u_2^\dual\in U_d^\dual,
\qquad u_i^\dual(u_j)=\delta_{ij},
\]
be the covectors dual to the basis of Lemma~\ref{lem:U-matrices}.  Since
$u_0=2^{4d-2}(\ve_{\cO_0}-\ve_{\cO_{\mathrm v}})$,
\begin{align*}
P_U(\ve_{\cO_0})&=2^{1-4d}u_0,&
P_U(\ve_{\cO_{\mathrm v}})&=-2^{1-4d}u_0,\\
P_U(\ve_{\cO_m})&=0\quad(1\le m\le2d),&
P_U(\ve_{\cO_{\mathrm s}})&=\ve_{\cO_{\mathrm s}},\qquad
P_U(\ve_{\cO_{\mathrm c}})=\ve_{\cO_{\mathrm c}}.
\end{align*}
If $F=\sum_{\cO}f_{\cO}\ve_{\cO}\in M_k(\rho_d;\cE_d^{\mathrm{w}})$, it follows that
\[
u_0^\dual(P_UF)=2^{1-4d}(f_{\cO_0}-f_{\cO_{\mathrm v}})=O(q^{-1}).
\]
The other two coordinate functions are the $\cO_{\mathrm s}$- and $\cO_{\mathrm c}$-coefficients, in the order determined by the parity of $d$.  Define
\begin{equation}\label{eq:EU}
\cE_d^U=
\begin{cases}
\{(u_0^\dual,-1),(u_1^\dual,-d/2),(u_2^\dual,-(d+1)/2)\},&d\text{ even},\\[2mm]
\{(u_0^\dual,-1),(u_1^\dual,-(d+1)/2),(u_2^\dual,-d/2)\},&d\text{ odd}.
\end{cases}
\end{equation}
In particular,
\[
\tr\cE_d^U=-d-\frac32,\qquad
P_U:M_k(\rho_d;\cE_d^{\mathrm{w}})\longrightarrow M_k(\rho_d^U;\cE_d^U).
\]
Equivalently, $H=h_0u_0+h_1u_1+h_2u_2$ belongs to
$M_k(\rho_d^U;\cE_d^U)$ if and only if $h_0=O(q^{-1})$, and
\[
\begin{cases}
h_1=O(q^{-d/2}),\quad h_2=O(q^{-(d+1)/2}),&d\text{ even},\\
h_1=O(q^{-(d+1)/2}),\quad h_2=O(q^{-d/2}),&d\text{ odd}.
\end{cases}
\]

We compute the conditions in \eqref{eq:EU} on the two summands of $U_d$ directly.  A form taking values in $U_d^{(1)}=\C(u_0+u_1)$ is $g(u_0+u_1)$, whose coordinates are $(g,g,0)$.  Since $d\ge2$, the $u_0$ condition is at least as strong as the $u_1$ condition, so the single required bound is $g=O(q^{-1})$.  A form taking values in
$U_d^{(2)}=\Span\{2u_0-u_1,u_2\}$ has the form
\[
f(2u_0-u_1)+h u_2,
\]
with coordinates $(2f,-f,h)$.  Its conditions are therefore
\[
f=O(q^{-1}),\qquad
h=
\begin{cases}
O(q^{-(d+1)/2}),&d\text{ even},\\
O(q^{-d/2}),&d\text{ odd}.
\end{cases}
\]
In the notation of Theorem~\ref{thm:rank3-classification}, using the eigenvectors
$u_0+u_1$, $2u_0-u_1$, and $u_2$ in place of $v_1,v_2,v_3$, this gives
\[
x_1=x_2=-1,\qquad
x_3=
\begin{cases}
-(d+1)/2,&d\text{ even},\\
-d/2,&d\text{ odd}.
\end{cases}
\]

The restricted growth conditions above now determine the generator weights
of the complementary module.

\begin{proposition}\label{prop:U-weights}
The module $M(\rho_d^U;\cE_d^U)$ is free with generator weights
\[
{-4d-8,\quad-4d-6,\quad-4d-4.}
\]
\end{proposition}

\begin{proof}
In the notation of Theorem~\ref{thm:rank3-classification}, put
$s=\tr\cE_d^U=-d-3/2$ and $A=12x_1=-12$.
For $d=2$, the restricted exponent sums add to $s$.
The split case gives weight $-12$ on the trivial summand and weights $-16,-14$ on the standard summand.

For $d=4$, we have
\[
x_2+x_3=-\tfrac72,\qquad B=6(x_2+x_3)-1=-22=4s.
\]
Here $A>4s+2$.
Case~(2) of Theorem~\ref{thm:rank3-classification} gives weights $-24,-22,-20$.

If $d\ge3$ is odd, then $B=-3d-7>4s$.
If $d\ge6$ is even, then $B=-3d-10>4s$.
In both cases $A>4s+2$, so case~(3) applies.
Its weights are
\[
4s-2=-4d-8,\qquad4s=-4d-6,\qquad4s+2=-4d-4.
\]
All cases give the stated list.
\end{proof}

The preceding calculation gives a well-defined map
\[
P_U:M(\rho_d;\cE_d^{\mathrm{w}})\longrightarrow M(\rho_d^U;\cE_d^U).
\]
\begin{theorem}\label{thm:exact-sequence}
For every $d\ge2$ there is an exact sequence of graded $\M$-modules
\[
0\longrightarrow
M(\rho_d^{\mathrm{cyc}};\cE_d^{\mathrm{cyc}})
\longrightarrow
M(\rho_d;\cE_d^{\mathrm{w}})
\stackrel{P_U}{\longrightarrow}
M(\rho_d^U;\cE_d^U)
\longrightarrow0.
\]
Consequently, the $\M$-module $J_{*,L_d}^{\mathrm{w},G_d}$ is free of rank $2d+4$
with generator-weight multiset
\[
{\{-4d+2j:0\le j\le2d\}\sqcup\{-8,-6,-4\}.}
\]
\end{theorem}

\begin{proof}
If $P_UF=0$, then $F$ takes values in $V_d^{\mathrm{cyc}}$.
Its $\cO_0$- and $\cO_{\mathrm v}$-coefficients are equal.
The bound $O(1)$ at $\cO_0$ is stronger than $O(q^{-1})$ at $\cO_{\mathrm v}$.
The remaining conditions are those in \eqref{eq:Ecyc}.
Thus
\[
\ker P_U=M(\rho_d^{\mathrm{cyc}};\cE_d^{\mathrm{cyc}}).
\]

We prove surjectivity.
Take $F\in M_k(\rho_d^U;\cE_d^U)$ and regard it as a $\cV_d$-valued form through $U_d\hookrightarrow\cV_d$.
Its $\cO_0$-coefficient is
\[
2^{4d-2}u_0^\dual(F)=cq^{-1}+O(1),\qquad c\in\C.
\]
Here the exponents are integral, by the action of $T$ on $u_0$.
All other conditions in \eqref{eq:Ed-weak} hold.
In particular, its $\cO_{\mathrm v}$-coefficient is the negative of its $\cO_0$-coefficient, and the $\cO_{\mathrm v}$ bound allows $q^{-1}$.
If $c=0$, the included $\cV_d$-valued form satisfies \eqref{eq:Ed-weak} and projects to the original $F$.
Assume $c\ne0$.
We seek a form $H$ of weight $k$ taking values in $V_d^{\mathrm{cyc}}$ such that
\[
v_0^\dual(H)=cq^{-1}+O(1),\qquad
v_m^\dual(H)=O(q^{-m/2})\quad(1\le m\le2d).
\]
Then $F-H$ satisfies \eqref{eq:Ed-weak} and still projects to $F$.

Put $G_j=\mathcal D^j\Phi_d$ for $0\le j\le2d$.
Let $A$ be their matrix of lowest Fourier coefficients:
\[
A_{mj}=[q^{-m/2}]\,v_m^\dual(G_j),\qquad 0\le m,j\le2d.
\]
Corollary~\ref{cor:cyclic-basis} and Theorem~\ref{thm:basis-criterion} give
\[
\det(G_0,\ldots,G_{2d})=\lambda\eta^{24\tr\cE_d^{\mathrm{cyc}}},
\qquad\lambda\ne0.
\]
Multiply row $m$ of the component matrix by $q^{m/2}$.
The resulting determinant has constant term $\det A$.
The normalized right-hand side has constant term $\lambda$.
Thus $\det A=\lambda\ne0$.

By Proposition~\ref{prop:U-weights}, $k\ge-4d-8$.
Hence
\[
\ell_j:=k+12-\wt(G_j)=k+8d+12-2j\ge4.
\]
Each $\ell_j$ is even.
The forms $E_{\ell_j}G_j$ all have weight $k+12$ and the same lowest-coefficient matrix $A$.
Choose constants $b_0,\ldots,b_{2d}$ with
$A(b_0,\ldots,b_{2d})^{\mathsf t}=(c,0,\ldots,0)^{\mathsf t}$.
Then
\[
H_0:=\sum_{j=0}^{2d}b_jE_{\ell_j}G_j
=(c+O(q))v_0+\sum_{m=1}^{2d}O(q^{-m/2+1})v_m.
\]
Since $\Delta=q+O(q^2)$ has weight $12$ and no zeros on $\Hh$,
$H=H_0/\Delta$ has weight $k$ and the required Fourier bounds.
Therefore $F-H\in M_k(\rho_d;\cE_d^{\mathrm{w}})$ and $P_U(F-H)=F$.
This proves surjectivity.

The target has a homogeneous free basis by Proposition~\ref{prop:U-weights}.
Choose a homogeneous preimage of each basis vector.
Extending these choices $\M$-linearly gives a graded section of $P_U$.
Hence
\[
M(\rho_d;\cE_d^{\mathrm{w}})
\cong M(\rho_d^{\mathrm{cyc}};\cE_d^{\mathrm{cyc}})
\oplus M(\rho_d^U;\cE_d^U).
\]
Corollary~\ref{cor:cyclic-basis} and Proposition~\ref{prop:U-weights} give the two lists of generator weights.
Theta decomposition adds $4d$ to each weight.
This proves the weak assertion of Theorem~C.
\end{proof}

\subsection{The holomorphic case}

We determine the holomorphic generator weights by comparing the dimensions
of the weak and holomorphic spaces.
Proposition~\ref{prop:principal-parts} gives this comparison in positive
vector-valued weight. At weight zero, we obtain matching upper and lower
bounds for the dimension of the space of invariant vectors.

\begin{lemma}\label{lem:zero-weight-bound}
Suppose $\rho$ has finite image and
\[
F,\mathcal DF,\ldots,\mathcal D^{r-1}F
\]
is a basis of $M(\rho;\cE)$, with $\wt(F)=k_0$.  Assume also that every holomorphic weight-zero form for $\rho$ belongs to $M_0(\rho;\cE)$; this holds, for example, when all lower bounds in $\cE$ are nonpositive.  Then
\[
\dim M_0(\rho)
\le
\dim M_{-k_0-2(r-1)}.
\]
\end{lemma}

\begin{proof}
Take $c\in M_0(\rho)$ and put $K=\ker\rho$.
Since $\rho$ has finite image, $K$ has finite index.
Each component of $c$ is a scalar modular form of weight zero on $K$.
It is holomorphic at every cusp: the components of $c|_0\gamma=\rho(\gamma)c$ are linear combinations of those of $c$.
Each component therefore extends to a holomorphic function on the compact modular curve for $K$ and is constant.
The transformation law makes this constant vector $\rho$-invariant.
Conversely, every invariant constant vector is a holomorphic weight-zero form.
In particular, $\DD c=0$.

By hypothesis, $c$ has a unique expansion
\[
c=\sum_{j=0}^{r-1}a_j\DD^jF,\qquad a_j\in M_{-k_0-2j}.
\]
We show that $c\mapsto a_{r-1}$ is an injection into $M_{-k_0-2(r-1)}$.
For $r=1$, this follows directly from $c=a_0F$.
Assume $r\ge2$ and suppose $a_{r-1}=0$ but $c\ne0$.
Choose the largest $m$ with $a_m\ne0$.
Then $m\le r-2$, and differentiation gives
\[
0=\DD c=\sum_{j=0}^{m}\bigl((\DD a_j)\DD^jF+a_j\DD^{j+1}F\bigr).
\]
The coefficient of $\DD^{m+1}F$ is $a_m$.
Every other term has derivative order at most $m$.
Independence of the basis forces $a_m=0$, a contradiction.
This proves the dimension bound.
\end{proof}

\begin{proposition}\label{prop:cyclic-invariants}
For every $d\ge2$,
\[
\dim M_0(\rho_d^{\mathrm{cyc}})=\left\lfloor\frac d3\right\rfloor+1,
\qquad
\dim M_0(\rho_d)=\left\lfloor\frac d3\right\rfloor+2.
\]
\end{proposition}

\begin{proof}
Put $a=\lfloor d/3\rfloor$.
Corollary~\ref{cor:cyclic-basis} and the scalar dimension formula give
\begin{equation}\label{eq:weak-zero-dim}
\dim M_0(\rho_d^{\mathrm{cyc}};\cE_d^{\mathrm{cyc}})
=\sum_{j=0}^{2d}\dim M_{8d-2j}
=d(d+1)+a+1.
\end{equation}
In coordinate $v_m$, holomorphy requires the vanishing of $\lceil m/2\rceil$ negative Fourier coefficients.
There are therefore
\[
\sum_{m=0}^{2d}\left\lceil\frac m2\right\rceil=d(d+1)
\]
linear conditions.
They reduce the dimension by at most $d(d+1)$.
Hence $\dim M_0(\rho_d^{\mathrm{cyc}})\ge a+1$.

All bounds in \eqref{eq:Ecyc} are nonpositive.
Lemma~\ref{lem:zero-weight-bound} applies with $r=2d+1$ and $k_0=-8d$.
It gives
\[
\dim M_0(\rho_d^{\mathrm{cyc}})\le\dim M_{4d}=a+1.
\]
The upper and lower bounds agree.

The proof of Lemma~\ref{lem:zero-weight-bound} identifies holomorphic weight-zero forms with invariant constant vectors.
Thus \eqref{eq:representation-splitting} gives
\[
M_0(\rho_d)\cong M_0(\rho_d^{\mathrm{cyc}})\oplus M_0(\rho_d^U).
\]
By Lemma~\ref{lem:U-matrices}, $U_d$ contains one trivial representation and one nontrivial irreducible representation.
Hence $\dim M_0(\rho_d^U)=1$, which proves the second equality.
\end{proof}

\begin{remark}
Identity \eqref{eq:weak-zero-dim} follows immediately by summing
\[
\dim M_k=
\begin{cases}
\lfloor k/12\rfloor+1,&k\not\equiv2\pmod{12},\\
\lfloor k/12\rfloor,&k\equiv2\pmod{12},
\end{cases}
\]
over the even weights $4d,4d+2,\ldots,8d$.
\end{remark}

\begin{theorem}\label{thm:holomorphic-weights}
The $\M$-module of holomorphic $G_d$-invariant Jacobi forms has generator weights
\[
{
\begin{array}{c|c}
\text{weight}&\text{multiplicity}\\ \hline
4d&\lfloor d/3\rfloor+2\\
4d+2&d-\lfloor d/3\rfloor+1\\
4d+4&d-\lfloor d/3\rfloor+1\\
4d+6&\lfloor d/3\rfloor
\end{array}}
\]
Consequently,
\[
\dim J_{4d,L_d}^{G_d}=\lfloor d/3\rfloor+2.
\]
For every even $k>4d$,
\[
\dim J_{k,L_d}^{\mathrm{w},G_d}-\dim J_{k,L_d}^{G_d}=d^2+2d+2.
\]
\end{theorem}

\begin{proof}
For positive even $w$, the space $S_{2-w}(\rho_d^\dual)$ is zero because $\rho_d$ has finite image.
Proposition~\ref{prop:principal-parts} and \eqref{eq:Ed-weak} give
\begin{equation}\label{eq:Dfamily-principal-difference}
\begin{aligned}
&\dim M_w(\rho_d;\cE_d^{\mathrm{w}})-\dim M_w(\rho_d)\\
&\qquad=\sum_{m=0}^{2d}\left\lceil\frac m2\right\rceil
+1+\left\lceil\frac d2\right\rceil+\left\lceil\frac{d+1}{2}\right\rceil\\
&\qquad=d(d+1)+1+(d+1)=d^2+2d+2.
\end{aligned}
\end{equation}
Theta decomposition shifts weight $w$ to Jacobi weight $w+4d$.
This proves the dimension difference.

Put $a=\lfloor d/3\rfloor$ and $h_w=\dim M_w(\rho_d)$.
Proposition~\ref{prop:cyclic-invariants} gives $h_0=a+2$.
For $w=2,4,6$, compute the weak dimensions from Theorem~\ref{thm:exact-sequence} and subtract \eqref{eq:Dfamily-principal-difference}.
The scalar dimension formula gives
\[
\begin{array}{c|cccc}
w&0&2&4&6\\ \hline
h_w&a+2&d-a+1&d+3&d+a+3
\end{array}
\]

There are no forms of negative weight, since $\rho_d$ has finite image.
Theorem~\ref{thm:freeness} makes $M(\rho_d)$ free of rank $2d+4$.
Let $b_w$ be the number of generators of weight $w$.
Since $M_2=0$ and $\dim M_0=\dim M_4=\dim M_6=1$,
\[
h_0=b_0,\qquad h_2=b_2,\qquad h_4=b_4+b_0,\qquad h_6=b_6+b_2+b_0.
\]
Solving gives
\[
(b_0,b_2,b_4,b_6)=(a+2,d-a+1,d-a+1,a).
\]
These multiplicities sum to $2d+4$, so there are no further generators.
Adding $4d$ to each weight gives the stated Jacobi weights.
This proves the holomorphic assertion of Theorem~C.
\end{proof}

\section{Weyl invariant Jacobi forms of lattice index
\texorpdfstring{$D_{8d}^{+}(p)$, $p=3,5$}{D8d+(p), p=3,5}}
\label{sec:odd-indices}

We determine the holomorphic generator weights for lattice indices
$D_{8d}^{+}(3)$ and $D_{8d}^{+}(5)$. We first express the invariant
coefficient representations as sums of symmetric powers of
representations of dimensions at most three. We then prove that they
factor through $A_4$ and $A_5$, respectively, and compute their
irreducible multiplicities from the character tables. The standard
holomorphic cusp conditions respect these decompositions, so the
corresponding modules are direct sums of those attached to the
irreducible representations. Applying Section~\ref{sec:rankthree} and
adding the theta weight $4d$ gives the generator weights and their
multiplicities in Theorem~D.

Let $d\ge2$ and retain
$G_d=\OO(D_{8d}^+)=W(D_{8d}^+)=W(D_{8d})$ from \eqref{eq:Gd-Weyl}.
For $p=3,5$, let
\[
A_{d,p}=D_{8d}^+/pD_{8d}^+,\qquad
V_{d,p}=\C[A_{d,p}]^{G_d},\qquad
\rho_{d,p}=\rho_{D_{8d}^+(p)}^{G_d}.
\]
Unimodularity identifies the discriminant group of $D_{8d}^+(p)$ with
$A_{d,p}$ by $x+pD_{8d}^+\mapsto x/p+D_{8d}^+$.
Theorem~\ref{thm:invariant-theta} gives
\begin{equation}\label{eq:odd-theta}
J_{K,D_{8d}^+(p)}^{G_d}\cong M_{K-4d}(\rho_{d,p}),
\end{equation}
where the module on the right uses the standard holomorphic conditions.
Since $-I\in G_d$, every $G_d$-invariant Jacobi form has even weight.

Let $E_p=\C[\mathbb F_p]$, with basis $(\ve_a)_{a\in\mathbb F_p}$,
and put $\zeta_p=\e(1/p)$.  Choose $\bar2$ with
$2\bar2\equiv1\pmod p$, and define
\[
H_p\ve_a=p^{-1/2}\sum_b\zeta_p^{ab}\ve_b,\qquad
T_p\ve_a=\zeta_p^{-\bar2a^2}\ve_a,\qquad
P_p\ve_a=\ve_{-a}.
\]
The normalized Gauss sums are
\[
\gamma_p=p^{-1/2}\sum_a\zeta_p^{-\bar2a^2},
\qquad
\gamma_3=i,\qquad\gamma_5=-1.
\]
Indeed, $1+2\zeta_3=i\sqrt3$ and
$1+2\zeta_5^2+2\zeta_5^3=-\sqrt5$.
Set
\begin{equation}\label{eq:local-rep}
\sigma_p(S)=\gamma_p^{-1}H_p,\qquad
\sigma_p(T)=T_p.
\end{equation}
Let $E_p^\pm$ be the two eigenspaces of $P_p$ and let
$\sigma_p^\pm$ denote the restrictions of these operators.

\begin{proposition}\label{prop:odd-invariants}
We have a $G_d$-equivariant isomorphism $A_{d,p}\cong\mathbb F_p^{8d}$.
The operators \eqref{eq:local-rep} define a representation of
$\Gammaone$, and
\begin{equation}\label{eq:main-sym}
\rho_{d,p}\cong
\Sym^{8d}(\sigma_p^+)\oplus\Sym^{8d}(\sigma_p^-).
\end{equation}
\end{proposition}

\begin{proof}
We identify the discriminant group, compare the Weil operators, and take $G_d$-invariants.
Since $p$ is odd, reduction of integral and half-integral coordinates is defined in $\mathbb F_p$.
The map is
\[
x+pD_{8d}^+\longmapsto
\bigl(\bar2(2x_1\bmod p),\ldots,\bar2(2x_{8d}\bmod p)\bigr).
\]
Adding an element of $pD_{8d}^+$ does not change the image.
To prove surjectivity, lift a residue vector to $z\in\Z^{8d}$.
If $\sum_i z_i$ is odd, replace $z_1$ by $z_1+p$.
The residues do not change, and the adjusted vector lies in $D_{8d}$.
Both $D_{8d}^+/pD_{8d}^+$ and $\mathbb F_p^{8d}$ have $p^{8d}$ elements, so the map is an isomorphism.
Reduction commutes with coordinate permutations and sign changes.
Thus the isomorphism is $G_d$-equivariant.

Integral and half-integral vectors can represent the same class.
For example, $\frac{p+1}{2}(1,\ldots,1)\in D_{8d}$ has the same image as $h$, since
\[
\frac{p+1}{2}(1,\ldots,1)-h=ph\in pD_{8d}^+.
\]

The quadratic value of $x/p+D_{8d}^+$ is $(x,x)/(2p)$.
The corresponding theta function has $T$-eigenvalue $\e((x,x)/(2p))$.
Its theta coefficient has the reciprocal eigenvalue.
Since $(x,x)/2$ is an integer, its reduction modulo $p$ is $\bar2\sum_i x_i^2$.
The discriminant pairing is $(x,y)/p$.
Its numerator reduces to $\sum_i x_i y_i$.
Writing $x,y\in\mathbb F_p^{8d}$ for the reduced coordinates gives
\begin{equation}\label{eq:lattice-ST}
\rho_{D_{8d}^+(p)}(T)\ve_x
 =\zeta_p^{-\bar2\scriptsum_i x_i^2}\ve_x,\qquad
\rho_{D_{8d}^+(p)}(S)\ve_x
 =p^{-4d}\sum_y\zeta_p^{\scriptsum_i x_i y_i}\ve_y.
\end{equation}
The signature factor is $\e(d)=1$.

We now verify the modular group relations for \eqref{eq:local-rep}.
Direct summation gives $H_p^2=P_p$.
Completing the square gives
\[
(H_pT_p)^2_{c,a}
 =\frac1p\zeta_p^{-\bar2a^2}
       \sum_b\zeta_p^{-\bar2b^2+b(a+c)}
 =\frac{\gamma_p}{\sqrt p}\zeta_p^{ac+\bar2c^2}.
\]
Multiplying once more yields
\[
(H_pT_p)^3_{b,a}
 =\frac{\gamma_p}{p}\sum_c\zeta_p^{c(a+b)}
 =\gamma_p\,\mathbf1_{\{b=-a\}}.
\]
Since $\gamma_p^4=1$, we obtain
\[
\sigma_p(S)^4=I,\qquad
(\sigma_p(S)\sigma_p(T))^3=\sigma_p(S)^2.
\]
The operators therefore define a representation of $\Gammaone$.
They commute with $P_p$, so they preserve $E_p^\pm$.
Equation~\eqref{eq:lattice-ST} identifies the full coefficient representation with $\sigma_p^{\otimes8d}$.
The factors $\gamma_p^{-1}$ in $\sigma_p(S)$ multiply to $\gamma_p^{-8d}=1$.

Write $G_d=H\rtimes S_{8d}$, where $H$ changes the signs of an even number of coordinates.
Decompose $E_p^{\otimes8d}$ according to the choices of $E_p^+$ and $E_p^-$ in each factor.
Suppose a tensor component contains a factor of each kind.
Changing the signs in one coordinate of each kind acts by $-1$.
Thus this component has no $H$-invariant vectors.
The two components with all factors of the same kind are fixed by $H$.
Hence
\[
(E_p^{\otimes8d})^H=(E_p^+)^{\otimes8d}\oplus(E_p^-)^{\otimes8d}.
\]
Taking $S_{8d}$-invariants gives the symmetric powers in \eqref{eq:main-sym}.
\end{proof}

\begin{proposition}\label{prop:odd-quotients}
The invariant representation $\rho_{d,p}$ factors through $A_4$
for $p=3$ and through $A_5$ for $p=5$.
\end{proposition}

\begin{proof}
Put $A_p=A_4$ for $p=3$ and $A_p=A_5$ for $p=5$.
By \cite[Section~6.4]{CM80},
\[
A_p\cong\langle r,u\mid r^2=u^3=(ru)^p=1\rangle.
\]
Set $s=r$ and $t=ru$.
Then $st=u$, so equivalently
\[
A_p\cong\langle s,t\mid s^2=(st)^3=t^p=1\rangle.
\]
Composing permutations from right to left, choose
\[
\begin{array}{c|ccc}
p&s&t&st\\ \hline
3&(12)(34)&(243)&(123)\\
5&(12)(34)&(14352)&(135)
\end{array}
\]
These permutations satisfy the relations.
For $p=3$, conjugating $s$ by $st$ gives the other double transpositions.
Together with $st$, they generate $A_4$.
For $p=5$, the generated subgroup has order divisible by $2$, $3$, and $5$.
Its order is therefore $30$ or $60$.
The first possibility would give a nontrivial map $A_5\to\Z/2\Z$.
Such a map kills every three-cycle, but three-cycles generate $A_5$.
Hence the subgroup is $A_5$.

The presentation $\Gammaone=\langle S,T\mid S^4=I,\ (ST)^3=S^2\rangle$ gives a surjection
\[
\theta_p:\Gammaone\longrightarrow A_p,\qquad
\theta_p(S)=s,\quad\theta_p(T)=t.
\]
To prove that $\rho_{d,p}$ factors through $\theta_p$, we identify $\ker\theta_p$.
Let $N$ be the normal subgroup generated by $S^2=-I$ and $T^p$.
Both elements lie in $\ker\theta_p$.
Taking the quotient adds precisely these two relations, so
\[
\Gammaone/N
=\langle\overline S,\overline T\mid
\overline S^2=(\overline S\overline T)^3=\overline T^p=1\rangle.
\]
This is the presentation of $A_p$ above.
Thus the induced map $\Gammaone/N\to A_p$ is an isomorphism and $\ker\theta_p=N$.

Global negation acts trivially on $V_{d,p}$, so $\rho_{d,p}(S^2)=I$.
Equation~\eqref{eq:lattice-ST} gives $\rho_{d,p}(T^p)=I$.
Therefore $\rho_{d,p}$ kills $N=\ker\theta_p$ and factors through $A_p$.
\end{proof}

The standard holomorphic condition respects invariant direct sums.
On a $T$-eigenspace with eigenvalue $\e(\lambda)$,
it uses the unique exponent $0\le\lambda<1$.
Every equivariant projection commutes with $T$ and preserves this bound.
Consequently,
\begin{equation}\label{eq:canonical-splitting}
\rho=\bigoplus_i\rho_i
\quad\Longrightarrow\quad
M(\rho)\cong\bigoplus_i M(\rho_i).
\end{equation}

We shall also use
\begin{equation}\label{eq:sym-character}
\sum_{k\ge0}\chi_{\Sym^k U}(g)u^k
 =\frac1{\det(1-u\,g|_U)}.
\end{equation}
For eigenvalues $\lambda_1,\ldots,\lambda_a$ of $g$, the monomials
in $\Sym^k U$ have eigenvalues
$\lambda_1^{j_1}\cdots\lambda_a^{j_a}$ with $\sum_i j_i=k$.
Summing over all monomials and all $k\ge0$ gives \eqref{eq:sym-character}.

The general character-theoretic statements below are classical;
see \cite[Sections~1.3--1.4 and~2.1--2.4]{Serre77}.

\begin{proposition}\label{prop:odd-characters}
Every finite-dimensional complex representation of a finite group $G$
is a direct sum of irreducible representations. If $X$ has character $\chi_X$ and
$\psi$ is an irreducible character, its multiplicity in $X$ is
\[
\langle\chi_X,\psi\rangle_G
 =\frac1{|G|}\sum_{g\in G}\chi_X(g)\overline{\psi(g)}.
\]
A nonzero representation $X$ is irreducible if and only if
$\langle\chi_X,\chi_X\rangle_G=1$.
A list of pairwise nonisomorphic irreducible representations is complete if and
only if the sum of their squared dimensions is $|G|$.
Moreover,
\[
\chi_{\Sym^2 X}(g)
 =\tfrac12\bigl(\chi_X(g)^2+\chi_X(g^2)\bigr).
\]

For $A_4$, the classes $1,2A,3A,3B$ are represented by
$1,\theta_3(S),\theta_3(T),\theta_3(T^2)$.
With $\zeta=\e(1/3)$, its irreducible character table is as follows;
see \cite[Section~5.7]{Serre77}.
\begin{equation}\label{eq:a4-table}
\begin{array}{c|rrrr}
 &1&2A&3A&3B\\ \hline
 \text{class size}&1&3&4&4\\
 \mathbf1&1&1&1&1\\
 \chi&1&1&\zeta&\zeta^2\\
 \chi^2&1&1&\zeta^2&\zeta\\
 \varrho&3&-1&0&0
\end{array}
\end{equation}
For $A_5$, the classes $1,2A,3A,5A,5B$ are represented by
$1,\theta_5(S),\theta_5(ST),\theta_5(T),\theta_5(T^2)$.
Put $\varphi=(1+\sqrt5)/2$ and $\varphi'=(1-\sqrt5)/2$.
Its irreducible character table is as follows;
see \cite[Section~3.1, Exercise~3.5]{FH91}.
\begin{equation}\label{eq:a5-table}
\begin{array}{c|rrrrr}
 &1&2A&3A&5A&5B\\ \hline
 \text{class size}&1&15&20&12&12\\
 \mathbf1&1&1&1&1&1\\
 \varrho_3&3&-1&0&\varphi&\varphi'\\
 \varrho'_3&3&-1&0&\varphi'&\varphi\\
 \varrho_4&4&0&1&-1&-1\\
 \varrho_5&5&1&-1&0&0
\end{array}
\end{equation}
\end{proposition}

\subsection{Lattice index \texorpdfstring{$D_{8d}^{+}(3)$}{D8d+(3)}}

The representation $\rho_{d,3}$ is a direct sum of irreducible representations
of $A_4$, each of dimension one or three. We compute their multiplicities
and apply the results of Section~\ref{sec:rankthree} to determine the
generator weights of the holomorphic Jacobi module.

\begin{theorem}\label{thm:holomorphic-index-three}
For $j=0,1,2$, put
\[
a_j(d)=\left\lfloor\frac{2d-j}{3}\right\rfloor+1+
\mathbf1_{\{j\equiv2d\pmod3\}}.
\]
The last term is one when $j\equiv2d\pmod3$ and zero otherwise.
The module $J_{*,D_{8d}^+(3)}^{G_d}$ is free of rank $8d+2$,
with generator weights and multiplicities
\[
\begin{array}{c|ccccc}
\text{weight}&4d&4d+2&4d+4&4d+6&4d+8\\ \hline
\text{multiplicity}&a_0(d)&2d&2d+a_1(d)&2d&a_2(d)
\end{array}
\]
\end{theorem}

\begin{proof}
Put $\zeta=\e(1/3)$ and write $U=\sigma_3^+$, $\chi=\sigma_3^-$.
Use the bases $\ve_0,\ve_1+\ve_{-1}$ on $E_3^+$ and $\ve_1-\ve_{-1}$ on $E_3^-$.
Equation~\eqref{eq:local-rep} gives
\begin{equation}\label{eq:seed3}
U(S)=\frac{-i}{\sqrt3}\begin{pmatrix}1&2\\1&-1\end{pmatrix},\qquad
U(T)=\begin{pmatrix}1&0\\0&\zeta\end{pmatrix},\qquad
\chi(S)=1,\quad\chi(T)=\zeta.
\end{equation}
By Proposition~\ref{prop:odd-invariants},
\begin{equation}\label{eq:v3}
V_{d,3}\cong\Sym^{8d}U\oplus\chi^{8d}.
\end{equation}
This is an $A_4$-representation by Proposition~\ref{prop:odd-quotients}.
Its dimension is $8d+2$.
The eigenvalues of $U(S)$ are $i,-i$, so
\[
\chi_{\Sym^{8d}U}(S)=\sum_{j=0}^{8d}i^{8d-j}(-i)^j=1.
\]
The additional line has trace $1$ at $S$.
Thus $\chi_{V_{d,3}}(S)=2$.

The dimension and the trace at $S$ determine the multiplicity of $\varrho$.
Let $b$ be its multiplicity, and let $a_j$ denote the multiplicity of $\chi^j$ for $j=0,1,2$.
The character table \eqref{eq:a4-table} gives
\[
a_0+a_1+a_2+3b=8d+2,\qquad a_0+a_1+a_2-b=2.
\]
Subtracting gives $b=2d$.

We determine $a_0,a_1,a_2$ from the $T$-eigenvalues.
The monomials in $\Sym^{8d}U$ have $T$-eigenvalues $\zeta^j$ for $0\le j\le8d$.
The additional line in \eqref{eq:v3} has eigenvalue $\zeta^{8d}$.
Hence the multiplicity of $\zeta^j$ in $V_{d,3}$ is
\[
t_j=\left\lfloor\frac{8d-j}{3}\right\rfloor+1+
\mathbf1_{\{j\equiv8d\pmod3\}}\qquad(j=0,1,2).
\]
Each copy of $\varrho$ contributes one of each $T$-eigenvalue.
Removing these $2d$ copies gives
\[
a_j=t_j-2d=\left\lfloor\frac{2d-j}{3}\right\rfloor+1+
\mathbf1_{\{j\equiv2d\pmod3\}}.
\]
Therefore
\begin{equation}\label{eq:a4-decomposition}
V_{d,3}\cong\mathbf1^{a_0}\oplus\chi^{a_1}\oplus(\chi^2)^{a_2}\oplus\varrho^{2d}.
\end{equation}

We now translate \eqref{eq:a4-decomposition} into generator weights.
The canonical exponents of $\mathbf1,\chi,\chi^2$ are $0,1/3,2/3$.
Lemma~\ref{lem:rank-one} gives generator weights $0,4,8$, respectively.
The exponents of $\varrho$ are $0,1/3,2/3$, with sum one.
Lemma~\ref{lem:rank-three-irreducible} gives weights $2,4,6$.
By \eqref{eq:canonical-splitting}, the generator weights of $M(\rho_{d,3})$ have multiplicities
\[
\begin{array}{c|ccccc}
\text{weight}&0&2&4&6&8\\ \hline
\text{multiplicity}&a_0&2d&2d+a_1&2d&a_2
\end{array}
\]
Finally, \eqref{eq:odd-theta} adds $4d$ to each weight.
\end{proof}

\subsection{Lattice index \texorpdfstring{$D_{8d}^{+}(5)$}{D8d+(5)}}

We determine the generator weights for lattice index $D_{8d}^{+}(5)$
by decomposing $\rho_{d,5}$ into irreducible representations of $A_5$.
The modules $M(\varrho_3)$ and $M(\varrho'_3)$ both have generator weights
$2,4,6$, so only the sum of the multiplicities of $\varrho_3$ and $\varrho'_3$
is needed.

Put
\[
r=\frac{(8d+1)(8d+4)}2,\qquad s=4d+2.
\]
Define the periodic functions $\alpha_d$ and $\beta_d$ by
\[
\begin{array}{c|ccc}
d\bmod3&0&1&2\\ \hline
\alpha_d&2&0&-1
\end{array}
\qquad
\begin{array}{c|ccccc}
d\bmod5&0&1&2&3&4\\ \hline
\beta_d&4&-2&0&0&3
\end{array}
\]

The following lemma computes the character values needed for the irreducible multiplicities.

\begin{lemma}\label{lem:index-five-character}
The representation $V_{d,5}$ has dimension $r$, and its character satisfies
\begin{equation}\label{eq:fullchar5}
\chi_{V_{d,5}}(S)=s,\qquad
\chi_{V_{d,5}}(ST)=\alpha_d,\qquad
\chi_{V_{d,5}}(T)+\chi_{V_{d,5}}(T^2)=\beta_d.
\end{equation}
\end{lemma}

\begin{proof}
Put $\zeta=\e(1/5)$ and
\[
a=\zeta+\zeta^{-1}=\frac{\sqrt5-1}{2},\qquad
b=\zeta^2+\zeta^{-2}=-\frac{1+\sqrt5}{2}.
\]
Then $a+b=-1$ and $ab=-1$.
Use the basis $\ve_0,\ve_1+\ve_{-1},\ve_2+\ve_{-2}$ on $E_5^+$.
On $E_5^-$ use $\ve_1-\ve_{-1},\ve_2-\ve_{-2}$.
Write $W=\sigma_5^+$ and $U=\sigma_5^-$.
Their matrices are
\begin{align}\label{eq:seed5plus}
W(S)&=-\frac1{\sqrt5}\begin{pmatrix}1&2&2\\1&a&b\\1&b&a\end{pmatrix},&
W(T)&=\operatorname{diag}(1,\zeta^2,\zeta^3),\\
U(S)&=-\frac1{\sqrt5}\begin{pmatrix}c&e\\e&-c\end{pmatrix},&
U(T)&=\operatorname{diag}(\zeta^2,\zeta^3),\label{eq:seed5minus}
\end{align}
where $c=\zeta-\zeta^{-1}$ and $e=\zeta^2-\zeta^{-2}$.
In particular, $ce=-\sqrt5$.
Proposition~\ref{prop:odd-invariants} gives
\begin{equation}\label{eq:v5}
V_{d,5}\cong\Sym^{8d}W\oplus\Sym^{8d}U.
\end{equation}
Its dimension is
\[
r=\binom{8d+2}{2}+(8d+1)=\frac{(8d+1)(8d+4)}2.
\]
By Proposition~\ref{prop:odd-quotients}, this is an $A_5$-representation.
We compute its character using the classes in \eqref{eq:a5-table}.

The class $2A$ is represented by $S$.
The eigenvalues of $W(S)$ are $1,-1,-1$, since $W(S)^2=I$ and $\tr W(S)=-1$.
Those of $U(S)$ are $i,-i$, since $U(S)^2=-I$ and $\tr U(S)=0$.
Equation~\eqref{eq:sym-character} gives
\[
\chi_{\Sym^{8d}W}(S)=[u^{8d}]\frac1{(1-u)(1+u)^2}=4d+1,\qquad
\chi_{\Sym^{8d}U}(S)=[u^{8d}]\frac1{1+u^2}=1.
\]
The first equality follows from $\sum_{j=0}^{8d}(-1)^j(j+1)=4d+1$.
For the second, $4d$ is even.
Hence
\begin{equation}\label{eq:char2}
\chi_{V_{d,5}}(S)=s=4d+2.
\end{equation}

The class $3A$ is represented by $ST$.
We have $\tr W(ST)=-(1+ab)/\sqrt5=0$ and $W(ST)^3=I$.
Thus its eigenvalues are $1,\omega,\omega^2$, where $\omega=\e(1/3)$.
Also $\tr U(ST)=-ce/\sqrt5=1$ and $\det U(ST)=1$.
Its eigenvalues are $\e(1/6),\e(-1/6)$.
Therefore
\[
\chi_{\Sym^{8d}W}(ST)=[u^{8d}]\frac1{1-u^3}=\mathbf1_{\{3\mid d\}},\qquad
\chi_{\Sym^{8d}U}(ST)=\frac{\sin((8d+1)\pi/3)}{\sin(\pi/3)}.
\]
Adding the traces gives
\begin{equation}\label{eq:char3}
\chi_{V_{d,5}}(ST)=\alpha_d,\qquad (\alpha_d)_{d\bmod3}=(2,0,-1).
\end{equation}

The classes $5A$ and $5B$ are represented by $T$ and $T^2$.
A five-cycle is conjugate to its inverse in $A_5$, but its square belongs to the other class.
We compute the sum of the two character values.
For $g=T,T^2$, define
\[
F_g(x)=\sum_{N\ge0}\bigl(\chi_{\Sym^N W}(g)+\chi_{\Sym^N U}(g)\bigr)x^N.
\]
By \eqref{eq:sym-character} and the $T$-eigenvalues above,
\begin{align}\label{eq:char5-series}
F_T(x)&=\frac1{(1-x)(1-\zeta^2x)(1-\zeta^3x)}
       +\frac1{(1-\zeta^2x)(1-\zeta^3x)}\notag\\
&=\frac{(2-x)(1-ax+x^2)}{1-x^5}.
\end{align}
For $T^2$, replace $a$ by $b$.
Since $a+b=-1$,
\[
F_T(x)+F_{T^2}(x)
=\frac{(2-x)(2+x+2x^2)}{1-x^5}
=\frac{4+3x^2-2x^3}{1-x^5}.
\]
Extracting $[x^{8d}]$ gives
\[
\chi_{V_{d,5}}(T)+\chi_{V_{d,5}}(T^2)=\beta_d,
\]
since $8d\equiv3d\pmod5$.
\end{proof}

\begin{theorem}\label{thm:holomorphic-index-five}
The module $J_{*,D_{8d}^+(5)}^{G_d}$ is free of rank
$r=(8d+1)(8d+4)/2$.
Put $s=4d+2$, and define the periodic functions $\alpha_d$ and $\beta_d$ by
\[
(\alpha_d)_{d\bmod3}=(2,0,-1),\qquad
(\beta_d)_{d\bmod5}=(4,-2,0,0,3).
\]
Let $g_j$ denote the number of generators of weight $4d+j$.
The generator weights and their multiplicities are as follows:
\[
\begin{array}{c|l}
\text{weight}&\multicolumn{1}{|c}{\text{multiplicity}}\\ \hline
4d&(r+15s+20\alpha_d+12\beta_d)/60\\
4d+2&(11r-15s-20\alpha_d+12\beta_d)/60\\
4d+4&(r-\alpha_d)/3\\
4d+6&(19r-15s+20\alpha_d-12\beta_d)/60\\
4d+8&(9r+15s-12\beta_d)/60
\end{array}
\]
All these multiplicities are nonnegative integers.
\end{theorem}

\begin{proof}
Put $\zeta=\e(1/5)$.
For the irreducible representations, the canonical exponents and generator weights are
\begin{equation}\label{eq:a5-profiles}
\begin{array}{c|c|c}
\text{representation}&\text{canonical }T\text{-exponents}&\text{generator weights}\\ \hline
\mathbf1&0&0\\
\varrho_3&0,1/5,4/5&2,4,6\\
\varrho'_3&0,2/5,3/5&2,4,6\\
\varrho_4&1/5,2/5,3/5,4/5&4,6,6,8\\
\varrho_5&0,1/5,2/5,3/5,4/5&2,4,4,6,8
\end{array}
\end{equation}
For $\varrho_3$ and $\varrho'_3$, restrict \eqref{eq:a5-table} to $\langle\theta_5(T)\rangle$.
The eigenvalues are $1,\zeta,\zeta^4$ and $1,\zeta^2,\zeta^3$.
The representation $\varrho_4$ can be realized as the restriction
of the natural permutation representation of $A_5$ on $\C^5$
to the invariant hyperplane $x_1+\cdots+x_5=0$. It then has eigenvalues
$\zeta,\zeta^2,\zeta^3,\zeta^4$.
For $\varrho_5$, the character vanishes at every nonidentity power of $T$.
Its restriction to $\langle T\rangle$ is therefore the regular representation, with each fifth root occurring once.
The exponent sums are one for $\varrho_3,\varrho'_3$ and two for $\varrho_4,\varrho_5$.
Lemmas~\ref{lem:rank-one} and~\ref{lem:rank-three-irreducible} and Proposition~\ref{prop:rank-four-five} give the weights in \eqref{eq:a5-profiles}.

Let $m_1,m_4,m_5$ be the multiplicities of $\mathbf1,\varrho_4,\varrho_5$.
Let $m_{33}$ be the sum of the multiplicities of $\varrho_3$ and $\varrho'_3$.
These last two representations have the same generator weights.
The characters of $\mathbf1,\varrho_4,\varrho_5$ take equal values on $5A$ and $5B$.
So does $\chi_{\varrho_3}+\chi_{\varrho'_3}$.
Lemma~\ref{lem:index-five-character} and the character inner product in Proposition~\ref{prop:odd-characters} give
\begin{align}\label{eq:a5-mults}
m_1&=\frac{r+15s+20\alpha_d+12\beta_d}{60},&
m_{33}&=\frac{r-5s+2\beta_d}{10},\notag\\
m_4&=\frac{r+5\alpha_d-3\beta_d}{15},&
m_5&=\frac{r+3s-4\alpha_d}{12}.
\end{align}
These numbers are nonnegative integers because they count irreducible summands.

By \eqref{eq:canonical-splitting} and \eqref{eq:a5-profiles}, the generator multiplicities are
\[
\begin{aligned}
g_0&=m_1,&g_2&=m_{33}+m_5,&g_4&=m_{33}+m_4+2m_5,\\
g_6&=m_{33}+2m_4+m_5,&g_8&=m_4+m_5.
\end{aligned}
\]
Substituting \eqref{eq:a5-mults} gives the five formulas in the theorem.
Their sum is $r$, so no other generator weights occur.
Finally, \eqref{eq:odd-theta} adds $4d$ to each weight.
\end{proof}

For $d=2$, the underlying lattice is $D_{16}^{+}$.
Theorems~\ref{thm:holomorphic-index-three} and~\ref{thm:holomorphic-index-five}
give the following generator weights and multiplicities for the holomorphic
Jacobi modules of lattice indices $D_{16}^{+}(3)$ and $D_{16}^{+}(5)$.
These modules have ranks $18$ and $170$, respectively, and both have
generators of each weight $8,10,12,14,16$:
\[
\begin{array}{c|rrrrr|r}
&8&10&12&14&16&\text{rank}\\ \hline
p=3&2&4&7&4&1&18\\
p=5&5&29&57&51&28&170
\end{array}
\]


\begin{thebibliography}{99}

\bibitem{AG20}
D.~Adler and V.~Gritsenko,
\emph{The $D_8$-tower of weak Jacobi forms and applications},
J. Geom. Phys. \textbf{150} (2020), 103616.

\bibitem{Bo98}
R.~E. Borcherds,
\emph{Automorphic forms with singularities on Grassmannians},
Invent. Math. \textbf{132} (1998), no.~3, 491--562.

\bibitem{Bo99}
R.~E. Borcherds,
\emph{The Gross--Kohnen--Zagier theorem in higher dimensions},
Duke Math. J. \textbf{97} (1999), no.~2, 219--233;
correction, Duke Math. J. \textbf{105} (2000), no.~1, 183--184.

\bibitem{Br02}
J.~H. Bruinier,
\emph{Borcherds Products on $O(2,l)$ and Chern Classes of Heegner Divisors},
Lecture Notes in Mathematics, vol.~1780, Springer-Verlag, Berlin, 2002.

\bibitem{CF17}
L.~Candelori and C.~Franc,
\emph{Vector-valued modular forms and the modular orbifold of elliptic curves},
Int. J. Number Theory \textbf{13} (2017), no.~1, 39--63.

\bibitem{CHMY}
L.~Candelori, T.~Hartland, C.~Marks, and D.~Yepez,
\emph{Indecomposable vector-valued modular forms and periods of modular curves},
Res. Number Theory \textbf{4} (2018), no.~2, Article 17.

\bibitem{CM80}
H.~S.~M. Coxeter and W.~O.~J. Moser,
\emph{Generators and Relations for Discrete Groups},
4th ed., Ergebnisse der Mathematik und ihrer Grenzgebiete,
vol.~14, Springer-Verlag, Berlin, 1980.

\bibitem{DLM00}
C.~Dong, H.~Li, and G.~Mason,
\emph{Modular-invariance of trace functions in orbifold theory and generalized moonshine},
Commun. Math. Phys. \textbf{214} (2000), no.~1, 1--56.

\bibitem{EZ85}
M.~Eichler and D.~Zagier,
\emph{The Theory of Jacobi Forms},
Progress in Mathematics, vol.~55, Birkh\"auser, 1985.

\bibitem{FM16}
C.~Franc and G.~Mason,
\emph{Three-dimensional imprimitive representations of the modular group and their associated modular forms},
J. Number Theory \textbf{160} (2016), 186--214.

\bibitem{FM18}
C.~Franc and G.~Mason,
\emph{On the structure of modules of vector-valued modular forms},
Ramanujan J. \textbf{47} (2018), no.~1, 117--139.

\bibitem{FH91}
W.~Fulton and J.~Harris,
\emph{Representation Theory: A First Course},
Graduate Texts in Mathematics, vol.~129, Springer-Verlag, New York, 1991.

\bibitem{Gan14}
T.~Gannon,
\emph{The theory of vector-valued modular forms for the modular group},
in \emph{Conformal Field Theory, Automorphic Forms and Related Topics},
W.~Kohnen and R.~Weissauer (eds.),
Contrib. Math. Comput. Sci., vol.~8,
Springer, Berlin--Heidelberg, 2014, pp.~247--286.

\bibitem{Gri88}
V.~A. Gritsenko,
\emph{Fourier--Jacobi functions of $n$ variables},
Zap. Nauchn. Sem. LOMI \textbf{168} (1988), 32--44;
English transl., J. Soviet Math. \textbf{53} (1991), no.~3, 243--252.

\bibitem{KM03}
M.~I. Knopp and G.~Mason,
\emph{On vector-valued modular forms and their Fourier coefficients},
Acta Arith. \textbf{110} (2003), no.~2, 117--124.

\bibitem{KM04}
M.~I. Knopp and G.~Mason,
\emph{Vector-valued modular forms and Poincar\'e series},
Illinois J. Math. \textbf{48} (2004), no.~4, 1345--1366.

\bibitem{KMlog}
M.~Knopp and G.~Mason,
\emph{Logarithmic vector-valued modular forms},
Acta Arith. \textbf{147} (2011), no.~3, 261--282.

\bibitem{Marks11}
C.~Marks,
\emph{Irreducible vector-valued modular forms of dimension less than six},
Illinois J. Math. \textbf{55} (2011), no.~4, 1267--1297.

\bibitem{MM10}
C.~Marks and G.~Mason,
\emph{Structure of the module of vector-valued modular forms},
J. Lond. Math. Soc. (2) \textbf{82} (2010), no.~1, 32--48.

\bibitem{Mason07}
G.~Mason,
\emph{Vector-valued modular forms and linear differential operators},
Int. J. Number Theory \textbf{3} (2007), no.~3, 377--390.

\bibitem{Sakai26}
K.~Sakai,
\emph{The ring of Weyl invariant $E_8$ Jacobi forms},
Math. Z. \textbf{313} (2026), 71. 


\bibitem{Serre77}
J.-P. Serre,
\emph{Linear Representations of Finite Groups},
Graduate Texts in Mathematics, vol.~42, Springer-Verlag, New York, 1977.

\bibitem{SW22}
K.~Sun and H.~Wang,
\emph{Conway invariant Jacobi forms on the Leech lattice},
Forum Math. \textbf{34} (2022), no.~6, 1591--1619.

\bibitem{SW23}
K.~Sun and H.~Wang,
\emph{Weyl invariant $E_8$ Jacobi forms and $E$-strings},
Commun. Number Theory Phys. \textbf{17} (2023), no.~3, 553--582.

\bibitem{Wang21a}
H.~Wang,
\emph{Weyl invariant $E_8$ Jacobi forms},
Commun. Number Theory Phys. \textbf{15} (2021), no.~3, 517--573.

\bibitem{Wang21}
H.~Wang,
\emph{Weyl invariant Jacobi forms: a new approach},
Adv. Math. \textbf{384} (2021), Article 107752.

\bibitem{WW20}
H.~Wang and B.~Williams,
\emph{On some free algebras of orthogonal modular forms},
Adv. Math. \textbf{373} (2020), Article 107332.

\bibitem{WW23}
H.~Wang and B.~Williams,
\emph{On weak Jacobi forms of rank two},
J. Algebra \textbf{634} (2023), 722--754.

\bibitem{Wir92}
K.~Wirthm\"uller,
\emph{Root systems and Jacobi forms},
Compositio Math. \textbf{82} (1992), no.~3, 293--354.

\bibitem{Zhu96}
Y.~Zhu,
\emph{Modular invariance of characters of vertex operator algebras},
J. Amer. Math. Soc. \textbf{9} (1996), no.~1, 237--302.


\end{thebibliography}
\end{document}